\documentclass[11pt]{amsart}
\usepackage[margin=3cm]{geometry}
\usepackage{amssymb,amsmath, amsthm, mathtools,constants}
\usepackage[colorlinks=true, urlcolor=black,linkcolor=blue, citecolor=blue]{hyperref}
\usepackage{pgf,tikz}
\usetikzlibrary{arrows}
\usepackage{pgfplots}

\newtheorem{theorem}{Theorem}[section]
\newtheorem{lemma}{Lemma}[section]
\newtheorem{corollary}{Corollary}[section]
\newtheorem{definition}{Definition}[section]

\newtheorem{remark}{Remark}[section]
\newtheorem{proposition}{Proposition}[section]
\newtheorem{assumption}{Assumption}[section]

\newconstantfamily{smc}{symbol=c}
\newconstantfamily{lgc}{symbol=C}

\numberwithin{equation}{section}

\begin{document}

\title[Fluctuations in FPP on the square, tube, and torus]{Fluctuation bounds for first-passage percolation on the square, tube, and torus}

\author{Michael Damron, Christian Houdr\'e, Alperen \"Ozdemir}

\address{School of Mathematics, Georgia Institute of Technology, 686 Cherry St., Atlanta, GA 30332, USA}
%\author[1]{Bounghun Bock$^1$, Michael Damron}
%%\author[2]{Michael Damron}
%\author[2]{Jack Hanson}
%\affil[1]{Georgia Tech}
%%\affil[2]{Georgia Tech}
%\affil[2]{City University of New York, City College and Graduate Center}
%\date{}

\begin{abstract}
In first-passage percolation, one assigns i.i.d.~nonnegative weights $(t_e)$ to the edges of $\mathbb{Z}^d$ and studies the induced distance (passage time) $T(x,y)$ between vertices $x$ and $y$. It is known that for $d=2$, the fluctuations of $T(x,y)$ are at least order $\sqrt{\log |x-y|}$ under mild assumptions on $t_e$. We study the question of fluctuation lower bounds for $T_n$, the minimal passage time between two opposite sides of an $n$ by $n$ square. The main result is that, under a curvature assumption, this quantity has fluctuations at least of order $n^{1/8-\epsilon}$ for any $\epsilon>0$ when the $t_e$ are exponentially distributed. As previous arguments to bound the fluctuations of $T(x,y)$ only give a constant lower bound for those of $T_n$ (even assuming curvature), a different argument, representing $T_n$ as a minimum of cylinder passage times, and deriving more detailed information about the distribution of cylinder times using the Markov property, is developed. As a corollary, we obtain the first polynomial lower bounds on higher central moments of the discrete torus passage time, under the same curvature assumption.
\end{abstract}

\maketitle

\section{Introduction}

%{\color{green} TO DO
%\begin{enumerate}
%\item change constants
%\item figures
%\end{enumerate}
%}

In this paper, we study first-passage percolation (FPP) on the discrete square, tube, and torus. The square is defined as $B(n) = \mathbb{Z}^2 \cap [0,n]^2$ with edge set $E(n) = \{\{x,y\} : x,y \in B(n), |x-y|=1\}$. Let $(t_e)$ be a family of i.i.d.~exponential random variables with mean 1, indexed by all nearest neighbor edges of $\mathbb{Z}^2$. For $x,y \in B(n)$, we set
\begin{equation}\label{passtimepath}
T_n^{\text{sq}} (x,y)=\inf_{\Gamma:x \rightarrow y} T(\Gamma).
\end{equation}
Here, the minimum is taken over all vertex self-avoiding paths starting at $x$, ending at $y$, and taking edges in $E(n)$, and $T(\Gamma)=\sum_{e \in \Gamma} t_e.$ Then we define the passage time between the left and right sides of the square as
\begin{equation}\label{passtime}
T_n^{\text{sq}}=\inf_{ x \in B(n) \,:\, x \cdot \mathbf{e_1} = 0  \atop  y \in B(n)\, : \, y \cdot \mathbf{e_1} = n} T_n^{\text{sq}}(x,y),
\end{equation}
where $\mathbf{e_i}$ stands for the $i$-th coordinate vector.

To define the tube and torus passage times, we represent them on $\mathbb{Z}^2$ using periodic edge-weights. Let $E(n)^o$ be the set of edges $\{x,y\} \in E(n)$ such that at least one of $x$ or $y$ is in $B(n-1)$. Let $(t_e^{(n)})_{e \in E(n)^o}$ be a family of i.i.d.~exponential random variables with mean 1 assigned to the edges of $E(n)^o$ and extend the definition to all nearest neighbor edges of $\mathbb{Z}^2$ by defining $t_{e+nz}^{(n)} = t_e^{(n)}$ for $z \in \mathbb{Z}^2$. Observe that the distribution of $(t_e^{(n)})$ is invariant under integer translations of $\mathbb{Z}^2$. The tube will use the vertical strip $S(n)=\{x \in \mathbb{Z}^2 : 0 \leq x \cdot \mathbf{e}_1 \leq n\}$, and we accordingly define the passage time as
\begin{equation}\label{eq: tube_time}
T_n^{\text{tube}} = \inf_{x : x \cdot \mathbf{e}_1=0 \atop y :y \cdot \mathbf{e}_1=n} T_n^{\text{tube}}(x,y).
\end{equation}
Here, $T_n^{\text{tube}}(x,y)$ is defined for $x,y \in S(n)$ as
\[
T_n^{\text{tube}}(x,y) = \inf_{\Gamma:x \rightarrow y} T^{(n)}(\Gamma),
\]
with the infimum over all vertex self-avoiding paths connecting $x$ and $y$ and using vertices in $S(n)$. The term $T^{(n)}(\Gamma)$ is the passage time of $\Gamma$ computed with the weights $(t_e^{(n)})$: $T^{(n)}(\Gamma) = \sum_{e \in \Gamma} t_e^{(n)}$. Last, the torus passage time is defined as
\begin{equation}\label{eq: torus_time}
T_n^{\text{tor}} = \inf_{x : x \cdot \mathbf{e}_1 = 0} T^{(n)}(x,x+n\mathbf{e}_1),
\end{equation}
where $T^{(n)}(x,y)$ for $x,y \in \mathbb{Z}^2$ is the minimum of $T^{(n)}(\Gamma)$ over all paths connecting $x$ and $y$ (not necessarily staying in $S(n)$). As defined, the tube passage time is the same as the first passage time between the left and right sides of $B(n)$ (using weights $(t_e^{(n)})$) after we have identified the top and bottom sides, turning $B(n)$ into a tube. The torus passage time is the same as the first passage time among all paths that wind once in the $\mathbf{e}_1$ direction around a torus obtained from identifying the left and right sides of $B(n)$, as well as the top and bottom sides.

Our main goal is to find lower bounds for the \textit{fluctuations} of these variables. We use the following definition of fluctuations, similar to that taken in \cite{DHHX}.
 
\begin{definition} \label{fluct} Let $\{X_n\}_{n=1}^{\infty}$ be a sequence of real-valued random variables. $\{X_n\}_{n=1}^{\infty}$ is said to have fluctuations of at least order $f(n)$ if there exist reals $a_n,b_n$ and a number $c>0$ such that for all large $n$, $b_n - a_n \geq cf(n)$,
\begin{equation*}
\mathbf{P}(X_n \leq a_n) >c,  \, \textnormal{ and } \, \mathbf{P}(X_n \geq b_n) >c.
\end{equation*} 
\end{definition} 

 Observe that if $(X_n)$ has fluctuations of at least order $f(n)$, then $\liminf_{n \to \infty} \frac{\mathrm{Var}(X_n)}{f(n)^2} > 0$. However, the converse may fail. For example, the sequence defined by
\[ X_n= \begin{cases} 
      0 & \textnormal{ with probability } 1-\frac{1}{n} \\
      n & \textnormal{ with probability } \frac{1}{n}
   \end{cases}
\]
has diverging variance, but its fluctuations are not at least order constant.

 A logarithmic lower bound for the variance of the point-to-point minimal passage time $T(x,y)$ in the standard model (FPP on the infinite discrete lattice $\mathbb{Z}^2$) is well-known for a large class of distributions of edge weights; see \cite[Sec.~3.3]{ADH17} and \cite{BC,DHHX,PP} for fluctuation bounds. A lower bound of polynomial order can even be shown under the curvature assumption we make below. Both of these arguments use an inequality of the form
 \[
 \text{Var}~T(x,y) \geq c\sum_e \mathbf{P}(e \text{ is in an optimal path from }x \text{ to } y)^2,
 \]
and this sum  has at least logarithmic growth in $|x-y|$ (polynomial growth under the curvature assumption). Unfortunately, when one uses a similar bound for $T_n,T_n^{\text{tube}},$ or $T_n^{\text{tor}}$, one obtains a sum of constant order. For example, the probability in the sum corresponding to $T_n^{\text{tor}}$ has order $n^{-1}$ by translation invariance. The main difference is that $T(x,y)$ is a passage time between fixed points, whereas the others are passage times between large sets of vertices. There are currently no general methods to analyze the variance which do not rely on bounding terms of this form, and therefore there are currently no nonconstant lower bounds for the variance of these three variables, even assuming curvature. For this reason, we will use a different method, based on finer information about the distribution of cylinder passage times that comes from the memoryless property of exponential weights.
  
We first define the limiting shape to state the curvature assumption. Similar to \eqref{passtimepath}, we define the passage time between $x,y \in \mathbb{Z}^2$ using all nearest neighbor edges. It is
\begin{equation}
T(x,y)=\min_{\Gamma:x \rightarrow y} T(\Gamma),
\end{equation}
where the minimum is over all self-avoiding paths $\Gamma$ starting from vertex $x$ and ending at $y$, and $T(\Gamma)=\sum_{e \in \Gamma} t_e.$ By the subadditive ergodic theorem (see \cite[Thm.~2.1]{ADH17}), it can be shown that there exists a norm $g: \mathbf{R}^2 \rightarrow \mathbf{R}_+$ such that a.s.,
\begin{equation}\label{eq: g_def}
g(x)=\lim_{n \rightarrow \infty} \frac{T(0,nx)}{n} \text{ for all } x \in \mathbb{Z}^2.
\end{equation}
Next, we define 
\begin{equation*}
B(t)=\{x \in \mathbb{Z}^2 \, : \, T(0,x) \leq t\}
\end{equation*}
and set $\widetilde{B}(t)$ equal to the sum set $B(t) + [0,1)^2$. Considering the passage time in all directions simultaneously, the shape theorem \cite{CD81} states that there exists a deterministic, convex, compact set $\mathcal{B} \subset \mathbf{R}^2$ with nonempty interior and the symmetries of $\mathbb{Z}^2$ that fix the origin such that for all $\epsilon > 0,$
\begin{equation} \label{limshape}
\mathbf{P}((1-\epsilon) \mathcal{B} \subset \widetilde{B}(t)/t \subset (1+\epsilon) \mathcal{B} \, \, \textnormal{ for all large }t)=1.
\end{equation}
%under some mild moment conditions on the edge weights, which are satisfied by the exponential distribution. 
We can express the limit shape as
\begin{equation*}
 \mathcal{B}=\left\{ x \in \mathbf{R}^2 \, : \, g(x) \leq 1 \right\}.
\end{equation*} 
%We will also use a strengthened version of the limit shape theorem. From \cite{A97,DK16,T15}, there exists $C>0$ such that {\color{green} Is this actually used?}
%\begin{equation} \label{ratelimshape}
%\mathbf{P}((t-C\sqrt{t \log t}) \mathcal{B} \subset \widetilde{B}(t) \subset (t+C\sqrt{t \log t}) \mathcal{B} \, \, \textnormal{ for all large }t)=1.
%\end{equation}

Next, we state our curvature assumption, which is about the right extreme of the limit shape. It has not been verified for any edge-weight distribution, but it is strongly believed to hold; see \cite[Sec.~2.8]{ADH17}.
 
\begin{assumption}\label{curv} \textnormal{(Curvature assumption in the direction $\mathbf{e_1}$)} 
%There is a hyperplane $H_0$ perpendicular to $\mathbf{e_1}$ through the origin such that $H_0+\mathbf{e_1} /g(\mathbf{e_1})$ is a supporting hyperplane for $\mathcal{B}$ and t
There are constants $\epsilon_0,c_0 >0$ such that  for all $\beta \in (-\epsilon_0,\epsilon_0)$,
\begin{equation}\label{eq: curvature_inequality}
g(\mathbf{e_1}+\beta \mathbf{e}_2) - g(\mathbf{e}_1) \geq c_0 \beta^2.
\end{equation}
\end{assumption}

%We will use the implication of the assumption in \cite{N95} which essentialy says that the geodesics are confined to thin cylinders. Let $D(0,n \mathbf{e_1})$ be the maximum distance of any vertex in the geodesic to $\mathbf{e_1}$-axis. This definition and the following result below can be found in \cite{C13}. By Assumption \eqref{curv}, there exists $C_1, C_2 > 0 $ such that
%\begin{equation} \label{curvcorr}
%\mathbf{P}\left(D(0,n \mathbf{e_1}\right) \geq n^{\frac{3}{4}+\epsilon}) \leq C_1 e^{-n^{C_2}}.
%\end{equation}

Now we are ready to state our result. 

\begin{theorem}\label{main}
Let $T_n^{\textnormal{sq}}$ and $T_n^{\textnormal{tube}}$ be passage times defined by \eqref{passtime} and \eqref{eq: tube_time}, and let $\epsilon> 0.$ Under Assumption~\ref{curv}, $T_n^{\textnormal{sq}}$ and $T_n^{\textnormal{tube}}$ both have fluctuations of at least order  $n^{1/8-\epsilon}.$
\end{theorem}

\begin{remark}
If we assume instead that \eqref{eq: curvature_inequality} holds with the exponent 2 replaced by $\kappa \geq 1$, then the estimate of Thm.~\ref{main} changes from $n^{1/8-\epsilon}$ to $n^{1/(4\kappa) - \epsilon}$.
\end{remark}

The strategy for the proof of Theorem~\ref{main} is to split the tube (or square) into non-overlapping cylinders of length $n$ and height $n^\alpha$, where $\alpha>3/4$. We first show that, because of Assumption~\ref{curv}, any optimal path must with high probability be contained in one of these cylinders (or a shifted version of these cylinders). This reduces the problem to finding the order of fluctuations of $n^{1-\alpha}$ many independent cylinder passage times. Cylinder times have been studied in \cite{A15, CD, DHHX}, but those works only provide lower bounds for the fluctuations of order $n^{(1-\alpha)/2}$ for a single time. To extend this to a minimum of many cylinder times, we need much more precise information about their distributions, in particular estimates for their $1/n^{1-\alpha}$ quantile. This is the main contribution of the present article. Using the memoryless property of exponentials, we represent the cylinder process as a Richardson-type growth model and prove that, conditional on geometric information of the growth, the times satisfy an entropic central limit theorem with bounds on the rate of convergence. Consequently, we can (conditionally) couple the cylinder times to independent normal variables. In the appendix, we derive a result that bounds the fluctuations of independent normal variables from below by the fluctuations of i.i.d.~normal variables. From this bound, we conclude that the fluctuations of the minimum of cylinder times are at least order $n^{(1-\alpha)/2}/ \sqrt{\log n}$. See Sec.~\ref{sec: outline} for an outline of the argument.

The torus passage time has particular difficulties which do not allow an easy comparison with the passage times for the square and the tube. But we can show a lower bound for the higher moments of the torus passage time as a corollary to our main theorem. It is an open problem \cite[Question~16]{ADH17} to show a diverging lower bound on the variance of $T_n^\textnormal{tor}$, even under a curvature assumption. The corollary below shows that for any real $k > 12$, the $k$-th central moment diverges.

\begin{corollary}\label{corollary}
Let $\epsilon > 0$ and $T_n^{\textnormal{tor}}$ be the torus passage time as defined in \eqref{eq: torus_time}. Under Assumption~\ref{curv}, there exists $c>0$ such that for all $n$ and $k$,
\begin{equation*}
\mathbf{E}\left|T_n^{\text{tor}} - \mathbf{E}T_n^{\text{tor}} \right|^k \geq  c n^{\left( \frac{1}{8} - \epsilon\right)k - \frac{3}{2}}.
\end{equation*}
%for some $c>0$ independent of $n.$
\end{corollary}

\subsection{Outline of the paper}\label{sec: outline}

In the next section, we study optimal paths between points in our periodic environment $(t_e^{(n)})$. After showing that they exist in Lem.~\ref{lem: geo_existence}, we prove in Prop.~\ref{conf} that optimal paths for $T_n^\textnormal{tube}$ are contained in horizontal cylinders of height $n^\alpha$, for any $\alpha>3/4$, with high probability. Although such a statement is standard in the planar model with i.i.d.~weights (under Assumption~\ref{curv}), it will take work to establish it in the periodic environment. In Sec.~\ref{sec: cylinders}, we study the fluctuations of the passage time across cylinders. First, because of the Markov property, we can represent this time using a Richardson-type growth model, and in Sec.~\ref{sect:boundary}, we estimate various quantities (size of the boundary of the growth, and number of steps to reach the opposite side) associated to it. We use these bounds in Sec.~\ref{sec: conditional_clt} to prove that the passage time across a cylinder satisfies a conditional (entropic) CLT, with a total variation bound coming from the estimates from Sec.~\ref{sect:boundary}. Using this, in Sec.~\ref{sec: independent_cylinders}, we prove our main fluctuation result for the minimum of independent cylinder times, Prop.~\ref{prop: independent_fluctuations}. This fluctuation result is the main tool in the proofs of Thm.~\ref{main} and Cor.~\ref{corollary}, in Sec.~\ref{sec: proofs}. Finally, the appendix serves to relate the fluctuations of the minimum of i.i.d.~normal random variables to the fluctuations of the minimum of independent normal random variables with different means and variances. This result is in Thm.~\ref{thm: gaussian_fluctuations} and is an important ingredient in the proof of Prop.~\ref{prop: independent_fluctuations} back in Sec.~\ref{sec: independent_cylinders}.

\section{Coupling and confinement of geodesics}

In this section, we focus on the discrete tube and show that the geodesics are with high probability contained  in cylinders with height of order $n^{\alpha}$ for $\alpha > 3/4$. As mentioned, this statement is well-known on $\mathbb{Z}^2$ under the curvature assumption; see \cite[Thm.~6]{NP95}. To show it for our periodic environment, we will need to couple the periodic model with the full-plane model and use concentration estimates.

It will be useful to observe that although $T_n^{\text{tube}}$ is defined using only paths that remain in the strip $S(n)$ by using $T_n^{\text{tube}}(x,y)$, this restriction is not necessary. That is,
\begin{equation}\label{tubetor}
T_n^{\textnormal{tube}}=\min_{ x \,:\, x \cdot \mathbf{e_1} = 0  \atop  y \, : \, y \cdot \mathbf{e_1} = n} T^{(n)}(x,y).
\end{equation}
The inequality $\geq$ holds trivially. For the other inequality, any path $\Gamma$ between some $x$ and $y$ as above contains a segment $\Gamma'$ which uses only vertices in $S(n)$. Indeed, we can simply follow $\Gamma$ until it touches the set $\{z : z \cdot \mathbf{e}_1 = n\}$ first at some point $y_0$, set this to be the final point of $\Gamma'$, and let the initial point of $\Gamma'$ be the last intersection of $\Gamma$ with $\{z : z \cdot \mathbf{e}_1 = 0\}$ before it touches $y_0$. Then $T^{(n)}(\Gamma) \geq T^{(n)}(\Gamma') \geq T_n^{\textnormal{tube}}$. Taking minimum over $\Gamma$ gives the inequality $\leq$ in \eqref{tubetor}.

To state geodesic concentration on the tube, we first need to show that geodesics exist. For $x,y \in \mathbb{Z}^2$, we say a path $\Gamma$ from $x$ to $y$ is a geodesic for $T^{(n)}(x,y)$ if $T^{(n)}(x,y) = T^{(n)}(\Gamma)$. Similarly, a path $\Gamma$ from $\{x : x \cdot \mathbf{e}_1 = 0\}$ to $\{x : x \cdot \mathbf{e}_1 = n\}$ is a geodesic for $T_n^{\textnormal{tube}}$ if $T^{(n)}(\Gamma) = T_n^{\textnormal{tube}}$. In general, geodesics need not be unique. For instance, with positive probability, there are two geodesics for $T^{(n)}(0,n(\mathbf{e}_1+\mathbf{e}_2))$.
\begin{lemma}\label{lem: geo_existence}
For any $x,y \in \mathbb{Z}^2$, 
\[
\mathbf{P}(\text{there is a geodesic for } T^{(n)}(x,y)\text{ from } x \text{ to } y) = 1.
\]
Also
\[
\mathbf{P}(\text{there is a geodesic for } T_n^{\textnormal{tube}}) = 1.
\]
\end{lemma}
\begin{proof}
The argument is similar to that for \cite[Prop.~4.4]{ADH17}. Let 
\[
\rho^{(n)} = \inf\{T^{(n)}(\gamma) : \gamma \text{ is an infinite edge self-avoiding path from } 0\}.
\]
By definition of $(t_e^{(n)})$, we have $\inf_e t_e^{(n)} > 0$ a.s., and so for any infinite self-avoiding $\gamma$ from $0$, we have $T^{(n)}(\gamma) \geq (\inf_e t_e^{(n)} )\#\gamma = \infty$. This means that $\rho^{(n)} = \infty$ a.s.~The argument in \cite[Prop.~4.4]{ADH17} shows that for any outcome such that $\rho^{(n)} = \infty$, there is a geodesic for $T^{(n)}(x,y)$ for all $x,y \in \mathbb{Z}^2$. (The proof is deterministic, so it applies to $T^{(n)}$.) For this reason, we treat only the second statement in more detail.

By periodicity, we may restrict $x$ in the definition of $T_n^{\textnormal{tube}}$ to be in $B(n)$. Let $\sigma$ be the path that starts at the origin and moves $n$ steps along the positive $\mathbf{e}_1$-axis until it ends at $n\mathbf{e}_1$. Fix any outcome for which $\rho^{(n)} = \infty$. Using the fact that $\rho^{(n)} = \lim_{K \to \infty} T^{(n)}(0,\partial [-K,K]^2)$ (from \cite[Lem.~4.3]{ADH17}), we can choose $K>n$ such that any path $\pi$ from $B(n)$ to a vertex in $([-K,K]^2)^c$ satisfies $T^{(n)}(\pi) > T^{(n)}(\sigma)$. Therefore any path that starts in $B(n)$ and leaves $[-K,K]^2$ cannot be a geodesic for $T_n^{\textnormal{tube}}$. This means that the minimum in the definition of $T_n^{\textnormal{tube}}$ (when we restrict $x$ to be in $B(n)$) is over a finite set, and there is a minimizer.
\end{proof}
\noindent

Next is the main result of the section, that for $\alpha > 3/4$, geodesics for $T_n^{\textnormal{tube}}$ are with high probability contained in horizontal strips of height $n^{\alpha}$.
\begin{proposition} \label{conf}
Let $\alpha>3/4$. There exists $b>0$ such that, for all large $n,$ with probability at least $1-e^{-n^b}$, the following holds. For any geodesic $\Gamma_n$ for $T_n^{\textnormal{tube}}$, all vertices on $\Gamma_n$ are contained in the strip $x_0 + \left( \mathbb{Z} \times [-n^{\alpha}, n^{\alpha}]\right),$ where $x_0$ is the initial point of $\Gamma_n$.
\end{proposition}

The proof will require the following concentration inequality for the passage time associated with periodic edge weights. Recall that $g$ was defined in \eqref{eq: g_def} using the passage time $T(x,y)$ over i.i.d.~exponential weights.
\begin{lemma}\label{concrep}
 Let $\epsilon>0$. There exist $\Cl[smc]{c: first},b > 0$ such that for all large $n$,
\[
\mathbf{P}\left( \text{for all }x,y \in \mathbb{Z}^2 \text{ with } \|x-y\|_{\infty} \leq \Cr{c: first} n, |T^{(n)}(x,y) - g(x-y)|\leq n^{\frac{1}{2}+ \epsilon}\right) \geq 1-e^{-n^{b}}.
\]
\end{lemma}
\begin{proof} By periodicity of $(t_e^{(n)})$, we may show that for large $n$,
\begin{align}
 \mathbf{P}( \text{for all }x,y \in \mathbb{Z}^2  \text{ with } x\in B(n) \text{ and } \|x-y\|_{\infty} \leq \Cr{c: first} n, |T^{(n)}(x,y) &- g(x-y)|\leq n^{\frac{1}{2} + \epsilon}) \nonumber \\
 &\geq 1 - e^{-n^{b}}. \label{eq: reduced_concrepeq}
 \end{align} 
 Furthermore, putting $x_0 =  \left( \lfloor \frac{n}{2} \rfloor, \lfloor \frac{n}{2} \rfloor\right)$, we may even show that
 \begin{equation}\label{eq: more_reduced_concrepeq}
 \mathbf{P}( \text{for all } y \in \mathbb{Z}^2 \text{ with } \|x_0-y\|_\infty \leq \Cr{c: first}  n, |T^{(n)}(x_0,y) - g(x_0-y)| \leq n^{\frac{1}{2} + \epsilon}) \geq 1 - e^{n^b}.
 \end{equation}
 This claim follows from translation invariance of the weights and a union bound. Specifically, assuming \eqref{eq: more_reduced_concrepeq} holds, the left side of \eqref{eq: reduced_concrepeq} is at least $1-n^2 e^{-n^b}$, and this implies \eqref{eq: reduced_concrepeq} if we replace $b$ with $b/2$.
 
 First we observe that \eqref{eq: more_reduced_concrepeq} holds if we replace $T^{(n)}(x_0,y)$ by $T(x_0,y)$; that is,
 \begin{equation} \label{eq: reduced_concrepeq_T}
 \mathbf{P}( \text{for all } y \in \mathbb{Z}^2 \text{ with } \|x_0-y\|_\infty \leq \Cr{c: first}  n, |T(x_0,y) - g(x_0-y)| \leq n^{\frac{1}{2} + \epsilon}) \geq 1 - e^{n^b}.
  \end{equation} 
 Inequality \eqref{eq: reduced_concrepeq_T} is standard and follows from two facts. First, Alexander has shown \cite[Thm.~3.2]{A97} that for some $\Cl[lgc]{c: alexander}>0$, we have $\mathbf{E}T(x,y) \leq g(x-y) + \Cr{c: alexander} \sqrt{|x-y|} \log (1+|x-y|)$ for all $x, y$. Second, Talagrand's inequality \cite[Prop.~8.3]{T95} states that for some $\Cl[lgc]{c: T1},\Cl[smc]{c: T2}>0$, we have
 \[
 \mathbf{P}(|T(x,y) - \mathbf{E}T(x,y)| \geq u) \leq \Cr{c: T1} \exp\left( - \Cr{c: T2} \min\left\{ \frac{u^2}{|x-y|},u\right\}\right)
 \]
 for all $x,y$ and all $u \geq 0$. If $\|x_0-y\|_\infty \leq \Cr{c: first} n$, then $|x_0-y| \leq \sqrt{2}\Cr{c: first} n$ and so by Alexander's bound,
 \begin{align*}
&\mathbf{P}\left(|T(x_0,y) - g(x_0-y)| \geq n^{\frac{1}{2}+\epsilon}\right) \\
\leq~& \mathbf{P}\left(|T(x_0,y) - \mathbf{E}T(x_0,y)| \geq n^{\frac{1}{2}+\epsilon} - \Cr{c: alexander} \sqrt{|x_0-y|}\log (1+|x_0-y|)\right) \\
\leq~& \mathbf{P}\left( |T(x_0,y) - \mathbf{E}T(x_0,y)| \geq n^{\frac{1}{2}+\frac{\epsilon}{2}}\right),
\end{align*}
 so long as $n$ is large. Applying Talagrand's inequality, this is bounded above for some $\Cl[smc]{c: new_times}>0$ by 
 \[
 \Cr{c: T1} \exp\left( - \Cr{c: T2} \min\left\{ \frac{n^{1+\epsilon}}{|x_0-y|}, n^{\frac{1}{2} + \frac{\epsilon}{2}}\right\}\right) \leq \Cr{c: T1} \exp\left( - \Cr{c: new_times} n^{\epsilon}\right).
 \]
 Therefore by a union bound, the left side of \eqref{eq: reduced_concrepeq_T} is bounded below by $1-\Cr{c: T1} (2\Cr{c: first}n+1)^2 e^{-\Cr{c: new_times}n^{\epsilon}}$. This gives \eqref{eq: reduced_concrepeq_T} with $b = \epsilon/2$.
 
To show \eqref{eq: more_reduced_concrepeq}, we need to compare $T$ and $T^{(n)}$. We will assume that the weights $(t_e)$ and $(t_e^{(n)})$ which are used in the definitions of $T$ and $T^{(n)}$ are coupled so that $t_e = t_e^{(n)}$ for all $e \in E_n^o$, and will prove that for large $n$,
 %for the classical growth model (edges weights are independent) or 3.12 in \cite{ADH17}.% 
\begin{equation} \label{repeqstd}
\mathbf{P}(\text{for all } y \in \mathbb{Z}^2  \text{ with }  \|x_0-y\|_{\infty} \leq \Cr{c: first}n,~T^{(n)}(x_0,y)=T(x_0,y)) \geq 1-e^{-n^{b}}.
\end{equation}
This, along with \eqref{eq: reduced_concrepeq_T}, will imply \eqref{eq: reduced_concrepeq}. To do this, let  $\Cr{c: first} < 1/2$ and consider an outcome in the complement: one for which there is a $y \in \mathbb{Z}^2$ with $\|x_0-y\|_\infty \leq \Cr{c: first} n$ and such that $T^{(n)}(x_0,y) \neq T(x_0,y)$. We claim that
\begin{equation}\label{eq: z_claim}
\text{there exists }z \in \mathbb{Z}^2 \text{ with }\|x_0-z\|_\infty = \lfloor n/2\rfloor \text{ such that }T(x_0,z) \leq T(x_0,y).
\end{equation}
To argue this, choose geodesics $\gamma$ and $\gamma^{(n)}$ for $T(x_0,y)$ and $T^{(n)}(x_0,y)$ respectively. If both paths use only edges in $E(n)^o$, then their passage times using $T$ or $T^{(n)}$ are the same, and so $T^{(n)}(x_0,y) = T(x_0,y)$, a contradiction. So at least one must use an edge outside $E(n)^o$, and therefore must contain a vertex $z$ with $\|x_0-z\|_\infty = \lfloor n/2 \rfloor$. First suppose that $\gamma$ contains such a $z$ and let $\gamma_z$ be the segment of $\gamma$ from $x_0$ to $z$. Then $T(x_0,z) \leq T(\gamma_z) \leq T(\gamma) = T(x_0,y)$, so \eqref{eq: z_claim} holds. The other possibility is that $\gamma^{(n)}$ contains such a $z$ but $\gamma$ does not. In this case, $\gamma$ must use only edges in $E(n)^o$, as must $\gamma_z^{(n)}$, the segment of $\gamma^{(n)}$ from $x_0$ to $z$, so long as we choose $z$ as the first such $z$ we find as we proceed along $\gamma^{(n)}$ from $x_0$ to $y$. Then
\[
T(x_0,z) \leq T(\gamma_z^{(n)}) = T^{(n)}(\gamma_z^{(n)}) \leq T^{(n)}(x_0,y) \leq T^{(n)}(\gamma) = T(\gamma) = T(x_0,y).
\]
This shows \eqref{eq: z_claim}.

Returning to \eqref{repeqstd}, statement \eqref{eq: z_claim} along with a union bound gives
\begin{align}
&\mathbf{P}(\text{there exists}~y \in \mathbb{Z}^2 \text{ with } \|x_0-y\|_\infty \leq \Cr{c: first}n,~T^{(n)}(x_0,y) \neq T(x_0,y))\nonumber \\
 \leq~& \sum_{y : \|x_0-y\|_\infty \leq \Cr{c: first} n \atop z : \|x_0-z\|_\infty = \lfloor \frac{n}{2}\rfloor} \mathbf{P}(T(x_0,z) \leq T(x_0,y)). \label{eq: exists_eqn}
\end{align}
If $|T(x_0,y) - g(x_0-y)|$ and $|T(x_0,z) - g(x_0-z)|$ were both at most $n^{3/4}$, then we would have
\[
T(x_0,z) - T(x_0,y) \geq g(x_0-z) - g(x_0-y) - 2n^{\frac{3}{4}},
\]
which is positive for large $n$ since $\|x_0-z\|_\infty = \lfloor n/2\rfloor$ and $\|x_0-y\|_\infty \leq \Cr{c: first} n$ (assuming we take $\Cr{c: first}$ fixed but small). Therefore
\[
\mathbf{P}(T(x_0,z) \leq T(x_0,y)) \leq \mathbf{P}\left(|T(x_0,y)-g(x_0-y)| \geq n^{\frac{3}{4}}\right) +  \mathbf{P}\left(|T(x_0,z)-g(x_0-z)| \geq n^{\frac{3}{4}}\right).
\]
Applying Talagrand's inequality, we obtain for some $\Cl[smc]{c: T_again}>0$
\[
\mathbf{P}\left(|T(x_0,y) - g(x_0-y)| \geq n^{\frac{3}{4}}\right) \leq \Cr{c: T1} \exp\left( - \Cr{c: T2} \min\left\{ \frac{n^{\frac{3}{2}}}{|x_0-y|}, n^{\frac{3}{4}}\right\}\right) \leq \Cr{c: T1} e^{-\Cr{c: T_again}n^{\frac{1}{2}}},
\]
with the same bound for $\mathbf{P}(|T(x_0,z)-g(x_0-z)| \geq n^{3/4})$. Putting this back in \eqref{eq: exists_eqn}, we find that the left side of \eqref{repeqstd} is bounded below by
\[
1 - 2\Cr{c: T1}\sum_{y : \|y-x_0\|_\infty \leq \Cr{c: first} n \atop z : \|x_0-z\|_\infty = \lfloor \frac{n}{2}\rfloor} \exp\left( -\Cr{c: T_again}n^{\frac{1}{2}}\right) \geq 1 - \Cl[lgc]{c: throwaway_1} n^3 \exp\left( - \Cr{c: T_again}n^{\frac{1}{2}}\right).
\]
This is bounded below by the right side of \eqref{repeqstd}, if we choose $b< 1/2$. This shows \eqref{repeqstd} and, along with \eqref{eq: reduced_concrepeq_T}, completes the proof.
\end{proof}

\begin{proof}[Proof of Prop.~\ref{conf}] Let $\alpha > 3/4$; we may additionally assume that $\alpha < 1$. Choose 
\begin{equation}\label{eq: epsilon_choice}
\epsilon \in \left(0,2\alpha - \frac{3}{2}\right)
\end{equation}
and fix an outcome in the event in the probability in Lem.~\ref{concrep} corresponding to this $\epsilon$. Let $\Gamma_n$ be a geodesic for $T_n^{\textnormal{tube}}$. We will show that all vertices in $\Gamma_n$ are contained in the strip $x_0 + \left( \mathbb{Z} \times [-\Cl[lgc]{c: tube_height} n^\alpha, \Cr{c: tube_height}n^\alpha]\right)$ for $ \Cr{c: tube_height} = 1+4/\Cr{c: first}$. (Here, $\Cr{c: first}$ is from Lem.~\ref{concrep}.) Because of the bound for the probability in Lem.~\ref{concrep}, this will prove Prop.~\ref{conf}, after slightly increasing $\alpha$.

To begin, we define a sequence of points $(x_i)$ on $\Gamma_n$ inductively as follows.
%{\color{red} START HERE} We use the Assumption \ref{curv} in the proof. We fix a realization of the edge weights and split the associated geodesic $\Gamma_n$ into segments as follows. 
%Suppose $\Gamma_n$ contains $N$ edges. Let $x_0$ be the initial and $x_N$ be the final point of $\Gamma_n.$ 
%and pick integers $r$ and $m$ to be the floor of $n^\alpha$ and $cn$, respectively. 
Let $x_0$ be the initial point of $\Gamma_n$. For $i \geq 0$, define the rectangle
\begin{equation*}
    \mathsf{C}_i= x_i + \left( [-\lfloor n^{\alpha}\rfloor, \lfloor \Cr{c: first}n \rfloor] \times [-\lfloor  n^{\alpha}\rfloor,\lfloor  n^{\alpha}\rfloor]\right),
\end{equation*}    
where $\Cr{c: first}$ is from Lem.~\ref{concrep}, its inner boundary
\[
\mathsf{B}_i = \{x \in \mathsf{C}_i : \exists y \in \mathbb{Z}^2 \setminus \mathsf{C}_i \text{ with } |x-y| \leq 1\},
\]
and its right inner boundary
\[
\mathsf{R}_i = \{ x \in \mathsf{B}_i : (x-x_i) \cdot \mathbf{e}_1 = \lfloor \Cr{c: first}n \rfloor\}.
\]
For $i \geq 1$, let $x_i$ be the first vertex of $\Gamma_n$ after $x_{i-1}$ that lies in $\mathsf{B}_i$, if one exists. If it does not exist, we set $x_i$ equal to the terminal point of $\Gamma_n$. Let $N$ be the first $i$ such that $x_i$ equals this terminal point. We claim that 
\begin{equation}\label{eq: box-claim}
\text{for all }i \in \{1,\ldots, N-1\},~x_i \in \mathsf{R}_{i-1}.
\end{equation}

To prove \eqref{eq: box-claim}, we use the linear functional $h: \mathbf{R}^2 \to \mathbf{R}$ defined by $h(z) = (g(\mathbf{e}_1) z) \cdot \mathbf{e}_1$. Observe that if we define $\bar{z} = z - 2(z \cdot \mathbf{e}_2)\mathbf{e}_2$, then by symmetry, $g(z) \geq (1/2) g(z+ \bar{z}) = (z \cdot \mathbf{e}_1) g( \mathbf{e}_1) = h(z)$. For certain $z$, though, we have a stronger bound. There exists $\Cl[smc]{c: curvature_loss}>0$ such that for large $n$,
\begin{equation}\label{eq: discrepancy_lower_bound}
g(z) - h(z) \geq \Cr{c: curvature_loss}n^{2\alpha-1} \text{ if } z \in \mathsf{B}_0 \setminus \mathsf{R}_0.
\end{equation}
We first prove this inequality, and then return to the proof of \eqref{eq: box-claim}. Let $z \in \mathsf{B}_0 \setminus \mathsf{R}_0$. Notice that if $z \cdot \mathbf{e}_1 \leq 0$, then $h(z) \leq 0$, so for some $\Cl[smc]{c: curvature_loss_2}>0$, if $n$ is large,
\begin{equation}\label{eq: case_1}
g(z) - h(z) \geq g(z) \geq \Cr{c: curvature_loss_2}\|z\|_\infty \geq \frac{\Cr{c: curvature_loss_2}}{4}n^\alpha.
\end{equation}
If $z \cdot \mathbf{e}_1 > 0$, then we must have $|z \cdot \mathbf{e}_2| \geq n^\alpha/2$ for large $n$, and so $|(z \cdot \mathbf{e}_2) / (z \cdot \mathbf{e}_1)| \geq n^{\alpha-1}/(2\Cr{c: first})$. Write
\begin{align}
g(z) - h(z) &= g((z \cdot \mathbf{e}_1)\mathbf{e}_1 + (z \cdot \mathbf{e}_2)\mathbf{e}_2) - g((z \cdot \mathbf{e}_1) \mathbf{e}_1)  \nonumber \\
&= (z \cdot \mathbf{e}_1) \left( g\left(\mathbf{e}_1 + \frac{z \cdot \mathbf{e}_2}{z \cdot \mathbf{e}_1} \mathbf{e}_2\right) - g(\mathbf{e}_1)\right). \label{eq: rice_a_roni}
\end{align}
If $|(z\cdot \mathbf{e}_2)/(z \cdot \mathbf{e}_1)| < \epsilon_0$, then we can use Assumption~\ref{curv} for the lower bound
\begin{equation}\label{eq: case_2}
g(z) - h(z) \geq c_0 (z\cdot \mathbf{e}_1) \left| \frac{z \cdot \mathbf{e}_2}{z \cdot \mathbf{e}_1}\right|^2 \geq c_0 \frac{n^{\alpha-1}}{2\Cr{c: first}} |z \cdot \mathbf{e}_2| \geq \frac{c_0}{4\Cr{c: first}} n^{2\alpha-1}.
\end{equation}
If $|z\cdot \mathbf{e}_2|/|z \cdot \mathbf{e}_1| \geq \epsilon_0$, then we use a modified curvature inequality: for $\beta$ such that $|\beta| \geq \epsilon_0$, we have
\begin{equation}\label{eq: modified_curv}
g(\mathbf{e}_1 + \beta \mathbf{e}_2) - g(\mathbf{e}_1) \geq c_0 \epsilon_0 |\beta|.
\end{equation}
To see why this holds, assume by symmetry that $\beta>0$ and set $\beta' = \epsilon_0$. By convexity of $g$, $g(\mathbf{e}_1 + \beta' \mathbf{e}_2) \leq (\beta'/\beta) g(\mathbf{e}_1 + \beta \mathbf{e}_2) + (1-\beta'/\beta)g(\mathbf{e}_1)$, and this gives $g(\mathbf{e}_1 + \beta \mathbf{e}_2) - g(\mathbf{e}_1) \geq (\beta/\beta') (g(\mathbf{e}_1 + \beta' \mathbf{e}_2) - g(\mathbf{e}_1))$. By Assumption~\ref{curv}, this implies
\[
g(\mathbf{e}_1 + \beta \mathbf{e}_2) - g(\mathbf{e}_1) \geq c_0 \frac{\beta}{\beta'}  (\beta')^2 = c_0 \epsilon_0 \beta,
\]
which is \eqref{eq: modified_curv}. Now, in the case that $|z \cdot \mathbf{e}_2|/|z \cdot \mathbf{e}_1| \geq \epsilon_0$, we apply \eqref{eq: modified_curv} in \eqref{eq: rice_a_roni} to obtain for large $n$
\begin{equation}\label{eq: case_3}
g(z) - h(z) \geq c_0 \epsilon_0 |z \cdot \mathbf{e}_2| \geq \frac{c_0 \epsilon_0}{2} n^\alpha.
\end{equation}
Combining the three cases \eqref{eq: case_1}, \eqref{eq: case_2}, and \eqref{eq: case_3}, and observing that $\alpha  > 2\alpha-1$, we conclude \eqref{eq: discrepancy_lower_bound}.

Having established \eqref{eq: discrepancy_lower_bound}, we can return to showing \eqref{eq: box-claim}. Let $I_R = \{ i = 1, \dots, N-1 : x_i \in \mathsf{R}_{i-1}\}$ and $J_R = \{1, \dots, N-1\} \setminus I_R$; we want to show that $\# J_R = 0$. Write
 \begin{align*}
 T_n^\textnormal{tube} -g(n\mathbf{e}_1)&= \sum_{i=1}^N ( T^{(n)}(x_{i-1},x_i) - h(x_{i-1}-x_i))\\
 	&= \sum_{i=1}^N (T^{(n)}(x_{i-1},x_i) - g(x_{i-1}-x_i)) + \sum_{i=1}^N ( g(x_{i-1}-x_i) - h(x_{i-1}-x_i)).
%         & \geq \sum_{i \in R} (h(x_{i+1}-x_i) - n^{1/2 + \epsilon}) + \sum_{i \in R^c} (h(x_{i+1}-x_i) + n^{1/2 + \delta}) \\
%         &=g(n\mathbf{e_1}) + |R^c| n^{1/2 + \delta} - |R| n^{1/2 + \epsilon}.  
 \end{align*}
Because our outcome is in the event described in Lem.~\ref{concrep} and $\|x_i-x_{i-1}\|_\infty \leq \Cr{c: first}n$, we have $|T^{(n)}(x_{i-1},x_i) - g(x_{i-1}-x_i)| \leq n^{1/2 + \epsilon}$ for all $i$. For $i \in J_R$, we apply \eqref{eq: discrepancy_lower_bound}, and for $i \notin J_R$, we use $g \geq h$. All together, we obtain
\[
T_n^\textnormal{tube} - g(n\mathbf{e}_1) \geq - n^{\frac{1}{2} + \epsilon}N + \Cr{c: curvature_loss}n^{2\alpha-1} \#J_R.
\]
To bound $T_n^{\textnormal{tube}}$ from above, we construct a path by starting at 0, moving to $\lfloor \Cr{c: first} n \rfloor \mathbf{e}_1$ using a geodesic for $T^{(n)}(0,\lfloor \Cr{c: first} n \rfloor)$, then moving to $2 \lfloor \Cr{c: first} n \rfloor$ using a geodesic for $T^{(n)}(\lfloor \Cr{c: first} n \rfloor, 2 \lfloor \Cr{c: first} n \rfloor)$, and so on, until we reach the largest multiple of $\lfloor \Cr{c: first} n \rfloor$ that is at most $n$. After this, we move to $n \mathbf{e}_1$. Using the fact that $g$ is additive along the $\mathbf{e}_1$-axis, we may apply the condition in Lem.~\ref{concrep} at each step to obtain $T_n^\textnormal{tube} - g(n\mathbf{e}_1) \leq (1/\Cr{c: first} + 1) n^{1/2 + \epsilon}$. Combining with the above produces
\begin{align*}
\left( \frac{1}{\Cr{c: first}} + 1\right) n^{\frac{1}{2} + \epsilon} &\geq -n^{\frac{1}{2} + \epsilon}N + \Cr{c: curvature_loss} n^{2\alpha-1} \#J_R \\
&= -n^{\frac{1}{2}+\epsilon} (\#I_R + 1) + \left( \Cr{c: curvature_loss}n^{2\alpha-1} - n^{\frac{1}{2}+\epsilon}\right) \#J_R.
\end{align*}
This implies
\begin{equation}\label{eq: one_bound}
\#I_R + \frac{1}{\Cr{c: first}} + 2 \geq \left( \Cr{c: curvature_loss}n^{2\alpha-\frac{3}{2}-\epsilon} - 1\right) \#J_R.
\end{equation}
 
 % Using the condition in Lemma \eqref{concrep} again, we have
% \begin{equation} \label{feq}
% 1+|R|  \geq |R^c| n^{\delta - \epsilon}.
% \end{equation}
To relate $\# I_R$ and $\# J_R$ in a different way, we look at the progression of each segment of $\Gamma_n$ in the $\mathbf{e_1}$ direction. For each $i \in J_R \cup \{N\}$, we have $(x_i - x_{i-1}) \cdot \mathbf{e}_1 \geq -n^\alpha$ and for each $i \in I_R$, we have $(x_i - x_{i-1}) \cdot \mathbf{e}_1 = \lfloor \Cr{c: first} n\rfloor$. Therefore
\begin{equation}\label{eq: in_e_1_direction}
 n = \sum_{i=1}^{N} (x_i-x_{i-1})\cdot \mathbf{e_1}  \geq \lfloor \Cr{c: first}n \rfloor \# I_R - n^{\alpha} (\# J_R +1).
\end{equation}
For large $n$, this gives $1 + n^{\alpha-1}(\#J_R+1) \geq (\Cr{c: first}/2)\#I_R$. Combining this with \eqref{eq: one_bound}, we find
\[
\frac{2}{\Cr{c: first}} \left( 1+n^{\alpha-1} (\#J_R + 1)\right) + \frac{1}{\Cr{c: first}} + 2 \geq \left( \Cr{c: curvature_loss}n^{2\alpha-\frac{3}{2}-\epsilon} - 1 \right) \#J_R.
\]
Recall that $\alpha<1$ but, because of \eqref{eq: epsilon_choice}, we have $2\alpha - 3/2 - \epsilon>0$. This inequality therefore cannot hold for large $n$ unless $\#J_R=0$. This proves \eqref{eq: box-claim}.
 
Now that we have shown \eqref{eq: box-claim}, we can quickly complete the proof of Prop.~\ref{conf}. Because $\#J_R=0$, \eqref{eq: in_e_1_direction} gives $n \geq \lfloor \Cr{c: first} n \rfloor \#I_R - n^\alpha$, with $\#I_R = N-1$, and so for large $n$, we have $n \geq (\Cr{c: first}/2)n (N-1) - n$, or $N \leq 1 + 4/\Cr{c: first}$. But then 
\[
\Gamma_n \subset \cup_{i=0}^{N-1} \mathsf{C}_i \subset x_0 + \left( \mathbb{Z} \times [-Nn^\alpha, Nn^\alpha]\right),
\]
and this shows Prop.~\ref{conf}.
\end{proof}

\section{Asymptotics for cylinder times}\label{sec: cylinders}

Because of the fact, from Prop.~\ref{conf}, that geodesics for $T_n^{\textnormal{tube}}$ are contained in cylinders, we are led to study passage times across cylinders. Our analysis will crucially rely on the Markov property that comes from our exponential weights. So we begin this section with a description of an alternate representation of the model.

Let $n,K \geq 1$ be integers and consider the cylinder $\mathsf{C}_{n,K} = [0,n]\times [0,K]$ with $K \leq n-1$. For vertices $x,y \in \mathsf{C}_{n,K}$, we define $T_{n,K}(x,y)$ as $\inf_{\Gamma : x\to y}T(\Gamma)$ where the infimum is over all paths with vertices in $\mathsf{C}_{n,K}$ from $x$ to $y$. (Here we can use the weights $(t_e)$ or $(t_e^{(n)})$ since they have the same distribution in $\mathsf{C}_{n,K}$, but for definiteness we now use $(t_e)$.) Then for $y \in \mathsf{C}_{n,K}$, we put
\[
T_{n,K}(y) = \inf_{x \in \mathsf{C}_{n,K} : x \cdot \mathbf{e}_1 = 0} T_{n,K}(x,y)
\]
and
\begin{equation}\label{eq: T_n_K_def}
T_{n,K} = \inf_{y \in \mathsf{C}_{n,K} : y \cdot \mathbf{e}_1 = n} T_{n,K}(y),
\end{equation}
The main observation behind the alternate representation is that, by the memoryless property of the exponential distribution, the sets $\{y \in \mathsf{C}_{n,K} : T_{n,K}(y) \leq t\}$ evolve as a Markov process as $t$ grows. Following the setup of \cite[p.~2]{PP}, we may build this process in two steps. First we grow a sequence of subgraphs of the cylinder as follows. Let $\mathcal{C}_0 = \{x \in \mathsf{C}_{n,K} : x \cdot \mathbf{e}_1 = 0\}$ with set of boundary edges $\mathcal{B}_0 = \{\{x,y\} : x \in \mathcal{C}_0, y \in \mathsf{C}_{n,K} \setminus \mathcal{C}_0, |x-y| = 1\}$. For $i \geq 0$, choose an edge $e_{i+1}$ uniformly from $\mathcal{B}_i$ and, writing $e_{i+1} = \{x_{i+1},y_{i+1}\}$, where $x_{i+1} \in \mathcal{C}_i$, set $\mathcal{C}_{i+1} = \mathcal{C}_i \cup \{y_{i+1}\}$ and $\mathcal{B}_{i+1} = \{\{x,y\} : x \in \mathcal{C}_{i+1}, y \in \mathsf{C}_{n,K} \setminus \mathcal{C}_{i+1}, |x-y|=1\}$. These sequences are defined until the value of $i = n(K+1)$ at which $\mathcal{C}_{n(K+1)} = \mathsf{C}_{n,K} \cap \mathbb{Z}^2$. In words, this is a random growth algorithm (a Richardson-type model \cite[Ch.~6]{ADH17}) in which we begin with a seed on the entire left side of the cylinder. At each timestep, the growth absorbs a uniformly chosen edge from its boundary in the cylinder. At some point we have touched the right side: define
\[
\mathcal{N} = \min\{i : y_i \cdot \mathbf{e}_1 = n\}.
\]
In the second stage of the process, we fix an outcome as above with sets $\mathcal{C}_i, \mathcal{B}_i$, and vertices $x_i,y_i$. Set $b_i = \# \mathcal{B}_{i-1}$ and let $X_1, X_2, \dots, X_\mathcal{N}$ be independent exponential random variables such that $\mathbf{E} X_i = b_i^{-1}$. Then
\begin{equation}\label{eq: growth_distribution}
T_{n,K} \text{ has the same distribution as } X_1 + \dots + X_\mathcal{N}.
\end{equation}
%Statement \eqref{eq: growth_distribution} is the analogue of the representation of the passage time $T$ in \cite[p.~2]{PP}, where it is stated ``The key observation is that the conditional joint distribution of the variables $T(V_n) - T(V_{n-1})$ given $\mathcal{F}$ is identical to a sequence of independent exponentials with means $1/Y_n$.'' In their notation, $V_n$ is the $n$-th vertex absorbed by the growth, $T(V_n)$ is the passage time from the origin to $V_n$, $\mathcal{F}$ is the $\sigma$-field generated by the $V_n$, and $Y_n$ is the number of edges on the boundary of the growth at the time $V_n$ is absorbed. 
As in \cite{PP}, ``This is an immediate consequence of the lack of memory of the exponential distribution and of the fact that the minimum of $n$ independent exponentials of mean 1 is an exponential of mean $1/n$.''  Another fact that follows directly from the representation is that
\begin{equation}\label{eq: p_1}
\mathcal{N} \text{ has the same distribution as }\#\{y \in \mathsf{C}_{n,K} : 0 < T_{n,K}(y) \leq T_{n,K}\}.
\end{equation}

In the following two subsections, we first prove bounds on the reals $b_i$, and then use them, along with an entropic central limit theorem, to bound the rate of convergence of $X_1 + \dots + X_\mathcal{N}$ (given the sequence $(b_i)$ and $\mathcal{N}$) to a standard normal distribution. This will allow us in the third subsection to estimate the fluctuations of the minimum of i.i.d.~copies of $T_{n,K}$.

\subsection{Boundary of the growth}\label{sect:boundary}

To estimate $\mathcal{N}$ and the $b_i$'s, we will use Kesten's lemma, which can be found in \cite[Lem.~4.5]{ADH17}.
\begin{lemma}\label{lem: kesten}
There exist $\Cl[smc]{c: kesten},a>0$ such that for all $k \geq 1$,
\[
\mathbf{P}(\exists \text{ vertex self-avoiding } \gamma \text{ from }0 \text{ with } \#\gamma \geq k \text{ but } T(\gamma) \leq ak) \leq e^{-\Cr{c: kesten}k}.
\]
\end{lemma}

First we give estimates for $\mathcal{N}$. The upper bound $\mathcal{N} \leq n(K+1)$ is immediate. For a lower bound, we have the following.
\begin{lemma}\label{lem: N_bounds}
There exists $\Cl[smc]{c: N_exponent}>0$ such that for all large $n$ and all $K \in [1,n-1]$,
\[
\mathbf{P}\left(\mathcal{N} \geq \frac{a}{2} nK\right) \geq 1-e^{-\Cr{c: N_exponent} n}.
\]
\end{lemma}
\begin{proof}
By \eqref{eq: p_1}, it suffices to show that for large $n$,
\begin{equation}\label{eq: equiv_to_show_now}
\mathbf{P}\left(\# \{y \in \mathsf{C}_{n,K} : 0 < T_{n,K}(y) \leq T_{n,K}\} < \frac{a}{2}nK\right) \leq e^{-\Cr{c: N_exponent}n}.
\end{equation}
%To do this, we use Kesten's lemma \cite{??}, which states that there exist $\Cr{c: kesten},a>0$ such that for all $k$,
%\[
%\mathbf{P}(\exists \text{ vertex self-avoiding } \gamma \text{ from }0 \text{ with } \#\gamma \geq n \text{ but } T(\gamma) \leq ak) \leq e^{-\Cr{c: kesten}k}.
%\]
If $T_{n,K} \leq an$, then there exists a vertex self-avoiding $\gamma$ starting from the left side of the cylinder with $\#\gamma \geq n$ but $T(\gamma) \leq an$. By Lemma~\ref{lem: kesten} and a union bound,
%\begin{equation}\label{eq: kierstens}
\[
\mathbf{P}(T_{n,K} \leq an) \leq (K+1) e^{-\Cr{c: kesten}n}.
\]
%\end{equation}
Therefore 
\begin{align}
&\mathbf{P}\left(\# \{y \in \mathsf{C}_{n,K} : 0 < T_{n,K}(y) \leq T_{n,K}\} < \frac{a}{2}nK\right)\nonumber \\
\leq~& (K+1)e^{-\Cr{c: kesten}n} + \mathbf{P}\left(\#\{y \in \mathsf{C}_{n,K} : 0 < T_{n,K}(y) \leq an\} < \frac{a}{2} nK\right). \label{eq: equiv_to_show_now_2}
\end{align}

We will prove that with high probability, the set of $y$ in \eqref{eq: equiv_to_show_now_2} contains $\left( [0,(a/2)n] \times [0,K]\right) \cap \mathbb{Z}^2$. Suppose that $y$ is in this latter set, and construct a deterministic path $\gamma_y$ by starting at $y$ and proceeding to $y-\mathbf{e}_1$, then to $y - 2\mathbf{e}_1$, and so on, in a straight line until we reach $(y \cdot \mathbf{e}_2) \mathbf{e}_2$. Then $\#\gamma_y$, the number of edges in $\gamma_y$, satisfies $\# \gamma_y \leq (a/2)n$ and we have $T_{n,K}(y) \leq T(\gamma_y)$, so for $\eta>0$, we can use Markov's inequality to obtain
\begin{equation}\label{eq: exponential_markov}
\begin{split}
\mathbf{P}(T_{n,K}(y) \geq an) \leq \mathbf{P}\left(e^{\eta T(\gamma_y)} \geq e^{\eta an}\right) &\leq e^{-\eta an} \left( \mathbf{E}e^{\eta t_e}\right)^{\#\gamma_y} \\
&\leq \left( e^{-\eta} \left( \mathbf{E}e^{\eta t_e}\right)^{\frac{1}{2}}\right)^{an} \\
&= \left( \frac{e^{-2\eta}}{1-\eta}\right)^{\frac{a}{2}n}.
\end{split}
\end{equation}
We fix $\eta$ to be small so that $b':= e^{-2\eta}/(1-\eta) < 1$ and conclude that by a union bound and the fact that $K \leq n$, 
\[
\mathbf{P}\left(T_{n,K}(y) \geq an \text{ for some } y \in \left[0, \frac{a}{2} n\right] \times [0,K] \right) \leq (K+1)\left( \frac{a}{2}n + 1\right) (b')^{\frac{a}{2}n} \leq e^{-\Cl[smc]{c: throwaway_10}n}
\]
for some $\Cr{c: throwaway_10}>0$, so long as $n$ is large. The event in the probability above is implied by the event in the probability in \eqref{eq: equiv_to_show_now_2}, so we obtain
\[
\mathbf{P}\left(\# \{y \in \mathsf{C}_{n,K} : 0 < T_{n,K}(y) \leq T_{n,K}\} < \frac{a}{2}nK\right) \leq (K+1)e^{-\Cr{c: kesten}n} + e^{-\Cr{c: throwaway_10}n} \leq e^{-\Cl[smc]{c: throwaway_11}n}
\]
for some $\Cr{c: throwaway_11}>0$ if $n$ is large. This shows \eqref{eq: equiv_to_show_now}.
\end{proof}

Next we estimate the boundary sizes $b_i$.
\begin{lemma}\label{lem: b_i_bounds}
For all $i = 1, \dots, \mathcal{N}$, we have $b_i \geq K+1$. Furthermore, there exists $\Cl[smc]{c: more_N_exponent}>0$ such that for all large $n$ and all $K \in [1,n-1]$,
\[
\mathbf{P}\left( \#\left\{i=1, \dots, \mathcal{N} : b_i \geq \frac{4}{a} K \right\} \geq \left( 1 - \frac{a}{2}\right) \mathcal{N}\right) \leq e^{-\Cr{c: more_N_exponent}n}.
\]
\end{lemma}
\begin{proof}
For a fixed $i$ and any $m=0, \dots, K$, choose $u_m \in \mathcal{C}_{i-1}$ to have $u_m \cdot \mathbf{e}_2 = m$ but with $u_m \cdot \mathbf{e}_1$ maximal. Then $\{u_m , u_m + \mathbf{e}_1\}$ is an edge of $\mathcal{B}_{i-1}$, so $b_i \geq K+1$. 

For the other bound, write $A$ for the event in the probability in the statement. We split $A$ according to the passage time $T_n^{\textnormal{tube}}$. Let $\Xi$ be the set of pairs $((d_i),M)$, where $d_i \in \mathbf{N}$ for $i \geq 1$ and $M \in \mathbf{N}$, such that all of the following hold:
\begin{itemize}
\item $M \leq n(K+1)$,
\item $d_i \geq K+1$ for all $i$, and
\item $\#\left\{i = 1, \dots, M : d_i \geq 4K/a \right\} \geq \left( 1 - a/2\right) M$.
\end{itemize}
Then
\begin{align*}
\mathbf{P}(A) &\leq \mathbf{P}(A, X_1 + \dots + X_{\mathcal{N}} \geq an) + \mathbf{P}(X_1 + \dots + X_\mathcal{N} \leq an) \\
&\leq \sum_{((d_i),M) \in \Xi} \mathbf{P}(X_1 + \dots + X_{\mathcal{N}} \geq an, (b_i) = (d_i), \mathcal{N} = M) + \mathbf{P}(T_{n,K} \leq an).
\end{align*}
As in the proof of Lemma~\ref{lem: N_bounds}, we have $\mathbf{P}(T_{n,K} \leq an) \leq e^{-\Cr{c: kesten}n}$ for large $n$. Therefore we obtain
\begin{align}
\mathbf{P}(A) &\leq e^{-\Cr{c: kesten}n}  \nonumber \\
&+ \sum_{((d_i),M) \in \Xi} \mathbf{P}(X_1 + \dots + X_{\mathcal{N}} \geq an \mid (b_i) = (d_i), \mathcal{N} = M)  \mathbf{P}((b_i) = (d_i), \mathcal{N} = M). \label{eq: split_time}
\end{align}

Conditional on $(b_i) = (d_i)$ and $\mathcal{N} = M$, the variables $X_i$ are independent exponentials with parameters $d_i$. So, for $(Y_i)$ that are i.i.d. exponential with mean 1, we have
\[
\mathbf{P}(X_1 + \dots + X_{\mathcal{N}} \geq an \mid (b_i) = (d_i), \mathcal{N} = M) = \mathbf{P}\left( \frac{Y_1}{d_1} + \dots + \frac{Y_M}{d_M} \geq an\right).
\]
We now split the sum depending on whether indices are in the set $S = \{i = 1, \dots, M : d_i \geq (4/a)K\}$. We obtain
\begin{align}
&\mathbf{P}(X_1 + \dots + X_{\mathcal{N}} \geq an \mid (b_i) = (d_i), \mathcal{N} = M) \nonumber \\
=~&\mathbf{P}\left( \sum_{i \in S} \frac{Y_i}{d_i} + \sum_{i \notin S} \frac{Y_i}{d_i} \geq an\right) \nonumber \\
\leq~&\mathbf{P}\left( \frac{a}{4K} \sum_{i \in S} Y_i \geq \frac{a}{3}n\right) + \mathbf{P}\left( \frac{1}{K+1} \sum_{i \notin S} Y_i \geq \frac{2a}{3}n\right) \nonumber \\
\leq~&\mathbf{P}\left( \sum_{i=1}^{n(K+1)} Y_i \geq \frac{4}{3} nK \right) + \mathbf{P}\left( \sum_{i=1}^{\lceil n(K+1)\frac{a}{2}\rceil} Y_i \geq \frac{2}{3} an(K+1)\right). \label{eq: almost_done_mr_man}
\end{align}
By standard large deviation estimates for i.i.d.~exponentials, there exists $\Cl[smc]{c: exponential_LD}>0$ such that for all $k \geq 1$, we have $\mathbf{P}(Y_1 + \dots + Y_k \geq 7k/6) \leq e^{-\Cr{c: exponential_LD}k}$. So for large $n$, we can bound \eqref{eq: almost_done_mr_man} by
\[
e^{-\Cr{c: exponential_LD}n(K+1)} + e^{-\Cr{c: exponential_LD}\lceil n(K+1) \frac{a}{2} \rceil} \leq e^{-\Cl[smc]{c: more_exponential_LD}n},
\]
for some $\Cr{c: more_exponential_LD}>0$. Plug this result back into \eqref{eq: split_time} to obtain
\[
\mathbf{P}(A) \leq e^{-\Cr{c: kesten}n} + \sum_{((d_i),M) \in \Xi}  e^{-\Cr{c: more_exponential_LD}n} \mathbf{P}((b_i) = (d_i), \mathcal{N} = M) \leq e^{-\Cr{c: kesten}n} + e^{-\Cr{c: more_exponential_LD}n}.
\]
For large $n$, this is bounded by $e^{-\Cr{c: more_N_exponent}n}$, if $\Cr{c: more_N_exponent} < \min\{\Cr{c: kesten},\Cr{c: more_exponential_LD}\}$. This completes the proof.
\end{proof}

\subsection{Conditional CLT}\label{sec: conditional_clt}

In this section, we prove a central limit theorem for $T_{n,K}$ conditional on the sequence $(b_i)$ and the number $\mathcal{N}$. The main tool is a theorem from \cite{rate04} which bounds the total variation distance between linear combinations of independent variables and a standard normal variable. To state it, we give some terminology. Recall that an exponential variable $X$ with parameter 1 satisfies a Poincar\'{e} inequality. 
%with constant $1/4$ \cite[Ex.~3.22]{BLM}. 
Namely, for any smooth $f:\mathbf{R} \rightarrow \mathbf{R},$
\begin{equation*}
  \textnormal{Var}(f(X)) \leq 4 \, \mathbf{E}(f'(X)^2).
 \end{equation*} 
Recall that the entropy of a random variable $X$ with density function $f$ is defined as 
\[
\textnormal{Ent}(X):= - \int_{\mathbf{R}} f(x) \log f(x)~\text{d}x.
\]
and that
\begin{equation} \label{entdiff}
\textnormal{Ent}(\textnormal{Exp}(\lambda))=1-\log \lambda \quad \textnormal{and} \quad \textnormal{Ent}(Z)=\frac{\log 2\pi +1}{2},
\end{equation}
where $Z$ is the standard normal random variable. Last, the total variation distance between two probability measures $\mu$ and $\nu$ is
\[
\textnormal{d}_{\textnormal{TV}} \left( \mu, \nu \right) = \sup_{B} \left| \mu(B) - \nu(B)\right|,
\]
where the supremum is over all Borel sets $B \subset \mathbf{R}$. The total variation distance between two random variables is defined as the total variation distance between their distributions. With these notations, we have the following estimate.
\begin{lemma}\label{clt_lem}
Let $\{W_i\}_{i=1}^N$ be independent copies of a random variable $W$ which satisfies a Poincar\'{e} inequality with some constant $c>0.$ Let $\{a_i\}_{i=1}^N$ be such that $\sum_{i=1}^N a_i^2 =1$ and let $S_N= \sum_{i=1}^N a_i W_i.$ Letting $Z$ be a standard normal random variable, we have
\begin{equation*}
\textnormal{d}_{\textnormal{TV}} \left (S_N - \mathbf{E}(S_N),  Z \right )^2 \leq 2 \frac{\sum_{i=1}^N a_i^{4}}{\frac{c}{2} + \left(1-\frac{c}{2}\right) \sum_{i=1}^N a_i^{4}} (\textnormal{Ent} (Z) - \textnormal{Ent} (W)).
\end{equation*}
\end{lemma}
\begin{proof}
This is \cite[Thm.~1]{rate04} combined with the inequality $\textnormal{d}_{\textnormal{TV}}(S_N - \mathbf{E}S_N,Z)^2 \leq 2[\textnormal{Ent}(S_N - \mathbf{E}S_N) - \textnormal{Ent}(Z)]$ in \cite[Eq.~(1)]{rate04}.
\end{proof}

We will apply the lemma to the variables $X_1, \dots, X_\mathcal{N}$, but to make them independent, we must condition on $(b_i)$ and $(\mathcal{N})$. For this purpose, we define the admissible set $\Upsilon$ of pairs $((d_i),M)$, where $d_i \in \mathbf{N}$ for $i \geq 1$ and $M \in \mathbf{N}$, by the following conditions:
\begin{enumerate}
\item $anK/2 \leq M \leq n(K+1)$,
\item $d_i \geq K+1$ for all $i$, and 
\item $\#\left\{i = 1, \dots, M : d_i \leq 4K/a \right\} \geq aM/2$.
\end{enumerate}
Summarizing the previous section, if we combine Lemmas~\ref{lem: N_bounds} and \ref{lem: b_i_bounds}, we find that for all large $n$ and all $K \in [1,n-1]$,
\begin{equation}\label{eq: upsilon_probability_bound}
\mathbf{P}(((b_i),\mathcal{N}) \in \Upsilon) \geq 1-e^{-\Cl[smc]{c: good_pairs}n}
\end{equation}
for $\Cr{c: good_pairs} = (1/2) \min\{\Cr{c: N_exponent},\Cr{c: more_N_exponent}\}$. Next we define the conditional distribution 
\[
\mu_{n,K}^{(d_i),M}(B) = \mathbf{P}\left(\frac{\sum_{i=1}^{\mathcal{N}} (X_i - b_i^{-1})}{\sqrt{ \sum_{i=1}^\mathcal{N} b_i^{-2}}} \in B ~\bigg|~ (b_i) = (d_i), \mathcal{N} = M\right) \text{ for Borel } B \subset \mathbf{R}.
\]
Of course, given $(b_i) = (d_i)$ and $\mathcal{N} = M$, the $X_i$'s are just independent exponentials with mean $b_i^{-1} = d_i^{-1}$. 

\begin{proposition}\label{prop: coupling}
For $((d_i),M) \in \Upsilon$, we have
\[
\textnormal{d}_{\textnormal{TV}}(\mu_{n,K}^{(d_i),M}, \mu_G)^2 \leq (\log 2\pi - 1)\frac{2^{15}}{a^8 n(K+1)},
\]
where $\mu_G$ is the standard Gaussian distribution. Consequently there exists a probability measure $\mathbf{Q}$ and random variables $U,Z$ on some space such that under $\mathbf{Q}$, $U$ has distribution $\mu_{n,K}^{(d_i),M}$, $Z$ is a standard Gaussian, and
\[
\mathbf{Q}(U \neq Z) \leq \sqrt{(\log 2\pi - 1)\frac{2^{15}}{a^8 n(K+1)}}.
\]
\end{proposition}
\begin{proof}
The second statement is standard and follows from the coupling representation of total variation distance.
% \cite[Lem.~8.1]{BLM}. 
To show the first, we apply Lem.~\ref{clt_lem}. Given $(b_i) = (d_i)$ and $\mathcal{N} = M$,
\[
\frac{\sum_{i=1}^{\mathcal{N}} (X_i - b_i^{-1})}{\sqrt{ \sum_{i=1}^\mathcal{N} b_i^{-2}}} = \frac{\sum_{i=1}^M \frac{Y_i}{d_i}}{\sqrt{\sum_{i=1}^M d_i^{-2}}} - \mathbf{E}\frac{\sum_{i=1}^M \frac{Y_i}{d_i}}{\sqrt{\sum_{i=1}^M d_i^{-2}}} \text{ in distribution},
\]
where the $Y_i$ are i.i.d.~exponential random variables with mean 1. So we set $W_i = Y_i$ and $a_i = d_i^{-1}/\sqrt{\sum_{i=1}^M d_i^{-2}}$ in the lemma. By items (1) and (2) in the definition of $\Upsilon$, 
\[
\sum_{i=1}^M d_i^{-4} \leq \frac{1}{(K+1)^4} \cdot n(K+1) = n(K+1)^{-3}.
\]
By items (1) and (3),
\begin{equation}\label{eq: variance_estimate}
\left( \sum_{i=1}^M d_i^{-2}\right)^2 \geq \left( \frac{a^2}{16K^2} \cdot \frac{aM}{2}\right)^2 \geq \left( \frac{a^4 n}{64 K}\right)^2 = \frac{a^8n^2}{64^2K^2}.
\end{equation}
Combining these produces
\[
\sum_{i=1}^M a_i^4 = \frac{\sum_{i=1}^M d_i^{-4}}{\left( \sum_{i=1}^M d_i^{-2}\right)^2} \leq \frac{n(K+1)^{-3}}{\left( \frac{a^4n}{64K}\right)^2} \leq \frac{64^2}{a^8 n(K+1)}.
\]
Using Poincar\'e constant $1/4$ in Lem.~\ref{clt_lem}, we obtain the following bound for the right side in the lemma:
\begin{align*}
2 \frac{\sum_{i=1}^M a_i^{4}}{\frac{c}{2} + \left(1-\frac{c}{2}\right) \sum_{i=1}^M a_i^{4}} (\textnormal{Ent} (Z) - \textnormal{Ent} (W)) &\leq \frac{\log 2\pi - 1}{2} \cdot \frac{4}{c} \sum_{i=1}^M a_i^4 \\
&\leq (\log 2\pi - 1) \frac{2^{15}}{a^8n(K+1)}.
\end{align*}
This completes the proof.
\end{proof}

\subsection{Fluctuation bounds for independent~cylinder times}\label{sec: independent_cylinders}

Here we use the results of the last two subsections to prove a fluctuation lower bound for the minimum of passage times across disjoint cylinders. For $n \geq 1$, pick integers $K_n^{(1)}, \dots, K_n^{(r_n)} \in [1,n]$ such that $\sum_{j=1}^{r_n} K_n^{(j)} = n$. Define cylinders as $\mathsf{C}^{(1)} = [0,n] \times [0, K_n^{(1)}-1]$ and for $j=2, \dots, r_n$,
\[
\mathsf{C}^{(j)} = [0,n] \times [K_n^{(1)} + \dots + K_n^{(j-1)}, K_n^{(1)} + \dots + K_n^{(j)}-1].
\]
For each $j$, we set $T^{(j)}$ to be the corresponding cylinder passage time. It is the minimal passage time of any path in $\mathsf{C}^{(j)}$ connecting the left and right sides, defined analogously to \eqref{eq: T_n_K_def}. Because the cylinders are disjoint, the $T^{(j)}$ are independent. Last, put 
\[
\mathcal{T}_n = \mathcal{T}_{n,(K_n^{(j)})} = \min\{T^{(1)}, \dots, T^{(r_n)}\}.
\]

\begin{proposition}\label{prop: independent_fluctuations}
Suppose that $n^{-1/2} \sum_{j=1}^{r_n} (K_n^{(j)})^{-1/2} \to 0$. Then  
\[
(\mathcal{T}_n) \text{ has fluctuations of at least order } \min_{j=1, \dots, r_n} \sqrt{\frac{n}{K_n^{(j)}(1+\log r_n)}}.
\]
\end{proposition}
\begin{proof}
We represent the $T^{(j)}$ as in \eqref{eq: growth_distribution}, obtaining boundary sequences $(b_i^{(j)})$ and reals $(\mathcal{N}^{(j)})$ such that the pairs $((b_i^{(j)}), \mathcal{N}^{(j)})$ are independent as $j$ varies. We also find variables $X_i^{(j)}$ for $j=1, \dots, r_n$ and $i=1, \dots, \mathcal{N}^{(j)}$ such that conditional on the pairs, the $X_i^{(j)}$ are independent exponentials with means $\mathbf{E}X_i^{j)} = b_i^{(j)}$. Last,
\[
\sum_{i=1}^{\mathcal{N}^{(j)}} X_i^{(j)} = T^{(j)} \text{ in distribution for all } j = 1, \dots, r_n,
\]
and so
\[
\min_{j=1, \dots, r_n} \sum_{i=1}^{\mathcal{N}^{(j)}} X_i^{(j)} = \mathcal{T}_n \text{ in distribution}.
\]

We define corresponding admissible pairs $\Upsilon^{(j)}$: they are those $((d_i^{(j)}), M^{(j)})$ such that
\begin{enumerate}
\item $anK^{(j)}/2 \leq M^{(j)} \leq n(K_n^{(j)}+1)$,
\item $d_i^{(j)} \geq K_n^{(j)}+1$ for all $i$, and
\item $\#\{i = 1, \dots, M^{(j)} : d_i^{(j)} \leq 4K_n^{(j)}/a\} \geq aM^{(j)}/2$.
\end{enumerate}
Then by \eqref{eq: upsilon_probability_bound}, for all large $n$, and all choices of the $K_n^{(j)}$ as above,
\begin{equation}\label{eq: final_good_set_bound}
\mathbf{P}((b_i^{(j)}),\mathcal{N}^{(j)}) \in \Upsilon^{(j)} \text{ for all } j = 1, \dots, r_n) \geq 1 - r_ne^{-\Cr{c: good_pairs}n}.
\end{equation}

With these definitions, we compute for any Borel set $B \subset \mathbf{R}$
\begin{equation}\label{eq: initial_decomposition}
\begin{split}
&\mathbf{P}(\mathcal{T}_n \in B) \\
\geq~& \sum_{\Upsilon^{(1)} \times \dots \times \Upsilon^{(r_n)}} \bigg[\mathbf{P}\left(\min_{j=1, \dots, r_n} \sum_{i=1}^{\mathcal{N}^{(j)}} X_i^{(j)} \in B ~\bigg|~ ((b_i^{(j)}),\mathcal{N}^{(j)}) = ((d_i^{(j)}),M^{(j)}) \text{ for all }j\right) \\
\times~&\mathbf{P}(((b_i^{(j)}),\mathcal{N}^{(j)}) = ((d_i^{(j)}),M^{(j)}) \text{ for all }j) \bigg].
\end{split}
\end{equation}
Write $\mu^{(j)}  = \sum_{i=1}^{\mathcal{N}^{(j)}} (b_i^{(j)})^{-1}$ and $\sigma^{(j)} = \sqrt{\sum_{i=1}^{\mathcal{N}^{(j)}} (b_i^{(j)})^{-2}}$, and then set
\[
\mathcal{L}^{(j)} = \frac{\sum_{i=1}^{\mathcal{N}^{(j)}} X_i^{(j)} - \mu^{(j)}}{\sigma^{(j)}}.
\]
We observe that under the conditional distribution appearing in \eqref{eq: initial_decomposition}, the vector $\left( \mathcal{L}^{(j)}\right)_{j=1}^{r_n}$ has product distribution $\prod_{i=1}^{r_n} \mu_{n,K_n^{(j)}}^{(d_i^{(j)}),M^{(j)}}$. By combining Prop.~\ref{prop: coupling} with the elementary fact that $\textnormal{d}_{\textnormal{TV}}(\mu_1 \times \dots \times \mu_m, \nu_1 \times \dots \times \nu_m) \leq \sum_{i=1}^m \textnormal{d}_{\textnormal{TV}}(\mu_i,\nu_i)$ for probability measures $\mu_i, \nu_i$, we can find a probability measure $\mathbf{Q} = \mathbf{Q}\left(((d_i^{(j)}),M^{(j)})\right)$ and random variables $U_1, \dots, U_{r_n}, Z_1, \dots, Z_{r_n}$ defined on some space such that under $\mathbf{Q}$,
\begin{itemize}
\item $U_j$ has distribution $\mu_{n,K_n^{(j)}}^{(d_i^{(j)}),M^{(j)}}$,
\item $Z_j$ is a standard Gaussian, and
\item the pairs $(U_1,Z_1), \dots, (U_{r_n},Z_{r_n})$ are independent.
\end{itemize}
Furthermore, we have the estimate
%\begin{equation}\label{eq: coupling_inequality_all}
\[
\mathbf{Q}(U_j \neq Z_j\text{ for some } j = 1, \dots, r_n) \leq \sum_{j=1}^{r_n} \sqrt{(\log 2\pi - 1) \frac{2^{15}}{a^8n (K_n^{(j)}+1)}}.
\]
%\end{equation}

By the above remarks and this inequality, we can represent the probability in \eqref{eq: initial_decomposition} as
\begin{align}
&\mathbf{P}\left(\min_{j=1, \dots, r_n} \sum_{i=1}^{\mathcal{N}^{(j)}} X_i^{(j)} \in B ~\bigg|~ ((b_i^{(j)}),\mathcal{N}^{(j)}) = ((d_i^{(j)}),M^{(j)}) \text{ for all }j\right)\nonumber \\
=~& \mathbf{Q}\left( \min_{j=1, \dots, r_n} \left( \mu^{(j)} + \sigma^{(j)} U_j  \right) \in B\right)\nonumber \\
\geq~& \mathbf{Q}\left( \min_{j=1, \dots, r_n} \left( \mu^{(j)} + \sigma^{(j)}  Z_j \right)  \in B\right) - \mathbf{Q}(U_j \neq Z_j \text{ for some }j)\nonumber \\
\geq~&\mathbf{Q}\left( \min_{j=1, \dots, r_n} \left( \mu^{(j)} + \sigma^{(j)}  Z_j  \right) \in B\right) - \frac{\Cl[lgc]{c: big_constant}}{\sqrt{n}} \sum_{j=1}^{r_n} \frac{1}{\sqrt{K_n^{(j)}}}, \nonumber
\end{align}
where $\Cr{c: big_constant}>0$ is a constant. We plug this back into \eqref{eq: initial_decomposition} to find
\begin{align*}
&\liminf_{n \to \infty} \mathbf{P}(\mathcal{T}_n \in B) \nonumber \\
\geq~& \liminf_{n \to \infty} \sum_{\Upsilon^{(1)} \times \dots \times \Upsilon^{(r_n)}}\mathbf{Q}\left( \min_{j=1, \dots, r_n} \left( \mu^{(j)} + \sigma^{(j)}  Z_j  \right) \in B\right) \mathbf{P}(((b_i^{(j)}),\mathcal{N}^{(j)}) = ((d_i^{(j)}),M^{(j)}) \forall j) \nonumber \\
-~& \liminf_{n \to \infty} \frac{\Cr{c: big_constant}}{\sqrt{n}} \sum_{j=1}^{r_n} \frac{1}{\sqrt{K_n^{(j)}}} \sum_{\Upsilon^{(1)} \times \dots \times \Upsilon^{(r_n)}}\mathbf{P}(((b_i^{(j)}),\mathcal{N}^{(j)}) = ((d_i^{(j)}),M^{(j)}) ~\forall j).
\end{align*}
By our assumption in the statement of the proposition, the second term is zero, so we get
\begin{align}
&\liminf_{n \to \infty} \mathbf{P}(\mathcal{T}_n \in B) \nonumber \\
\geq~& \liminf_{n \to \infty} \sum_{\Upsilon^{(1)} \times \dots \times \Upsilon^{(r_n)}}\mathbf{Q}\left( \min_{j=1, \dots, r_n} \left( \mu^{(j)} + \sigma^{(j)}  Z_j  \right) \in B\right) \mathbf{P}(((b_i^{(j)}),\mathcal{N}^{(j)}) = ((d_i^{(j)}),M^{(j)}) \forall j) \label{eq: here_we_are_now}
\end{align}

By Thm.~\ref{thm: gaussian_fluctuations} in the appendix, if we write $\sigma = \sigma\left( ((d_i^{(j)}),M^{(j)})_{j=1}^{r_n} \right) = \min_{j=1, \dots, r_n} \sigma^{(j)}$, then we can find reals $a_n = a_n\left( ((d_i^{(j)}),M^{(j)})_{j=1}^{r_n} \right)$ and $b_n=b_n\left( ((d_i^{(j)}),M^{(j)})_{j=1}^{r_n}\right)$. and a universal constant $\Cl[smc]{c: from_appendix}>0$ such that 
\begin{equation}\label{eq: a_n_b_n_choice}
\begin{split}
&\mathbf{Q}\left( \min_{j=1, \dots, r_n} \left( \mu^{(j)} + \sigma^{(j)}  Z_j  \right) \leq a_n\right) \geq \Cr{c: from_appendix}, \\
&\mathbf{Q}\left( \min_{j=1, \dots, r_n} \left( \mu^{(j)} + \sigma^{(j)}  Z_j  \right) \geq b_n \right) \geq \Cr{c: from_appendix},
\end{split}
\end{equation}
and
\begin{equation}\label{eq: spread}
b_n - a_n = \frac{\sigma}{\sqrt{1+\log r_n}}.
\end{equation}
We would like to set $B = (-\infty,a_n]$ and then $B= [b_n,\infty)$ in \eqref{eq: here_we_are_now}, but since these reals depend on the pairs $((d_i^{(j)}), M^{(j)})$, we must replace them with reals that do not depend on the pairs.

To do this, we recall \eqref{eq: final_good_set_bound}, which implies that if $n$ is large enough, then
\[
\sum_{\Upsilon^{(1)} \times \dots \times \Upsilon^{(r_n)}} \mathbf{P}\left(((b_i^{(j)}),\mathcal{N}^{(j)}) = ((d_i^{(j)}),M^{(j)}) \forall j\right) = \mathbf{P}\left(((b_i^{(j)}),\mathcal{N}^{(j)}) \in \Upsilon^{(j)}~\forall j\right) \geq \frac{1}{2}.
\]
The quantities $a_n$ and $b_n$ are random variables when considered as functions of $((b_i^{(j)}),\mathcal{N}^{(j)})_{j=1}^{r_n}$. Because $K(x) = \mathbf{P}\left( a_n \leq x, ((b_i^{(j)}),\mathcal{N}^{(j)}) \in \Upsilon^{(j)}~\forall j\right)$ is a nondecreasing, right continuous function of $x$ with $\lim_{x \to -\infty} K(x) = 0$ and $\lim_{x \to \infty} K(x) \geq 1/2$, we can find a deterministic number $m_n^a$ such that both of the following hold:
\begin{align*}
&\mathbf{P}\left(a_n \leq m_n^a, ((b_i^{(j)}),\mathcal{N}^{(j)}) \in \Upsilon^{(j)}~\forall j\right) \geq \frac{1}{4} \\
&\mathbf{P}\left(a_n \geq m_n^a, ((b_i^{(j)}),\mathcal{N}^{(j)}) \in \Upsilon^{(j)}~\forall j\right) \geq \frac{1}{4}.
%&\mathbf{P}(b_n \leq m_n^b, ((b_i^{(j)}),\mathcal{N}^{(j)}) \in \Upsilon^{(j)}~\forall j) \geq \frac{1}{4} \\
%&\mathbf{P}(b_n \geq m_n^b, ((b_i^{(j)}),\mathcal{N}^{(j)}) \in \Upsilon^{(j)}~\forall j) \geq \frac{1}{4}.
\end{align*}
%Here we are writing simply $a_n$ for $a_n$ evaluated at the argument $((b_i^{(j)}), \mathcal{N}^{(j)})_{j=1}^{r_n}$. 
Observe that if we define
\[
\widetilde{\sigma} = \min\{ \sigma : ((d_i^{(j)}), M^{(j)}) \in \Upsilon^{(j)} \text{ for all }j\},
\]
which, by \eqref{eq: variance_estimate}, satisfies 
\begin{equation}\label{eq: tilde_variance_inequality}
\widetilde{\sigma} \geq \frac{a^2}{8} \min_{j=1, \dots, r_n} \sqrt{\frac{n}{K_n^{(j)}}},
\end{equation}
and
\begin{equation}\label{eq: median_choice}
m_n^b = m_n^a +  \frac{\widetilde{\sigma}}{\sqrt{1+\log r_n}},
\end{equation}
then by \eqref{eq: spread},
\begin{equation}\label{eq: b_n_inequality}
\mathbf{P}(b_n \geq m_n^b, ((b_i^{(j)}),\mathcal{N}^{(j)}) \in \Upsilon^{(j)}~\forall j)\geq \mathbf{P}(a_n \geq m_n^a, ((b_i^{(j)}),\mathcal{N}^{(j)}) \in \Upsilon^{(j)}~\forall j) \geq \frac{1}{4}.
\end{equation}
First we set $B = (-\infty, m_n^a]$ in \eqref{eq: here_we_are_now} and apply \eqref{eq: a_n_b_n_choice} to obtain
\begin{align}
&\sum_{\Upsilon^{(1)} \times \dots \times \Upsilon^{(r_n)}}\mathbf{Q}\left( \min_{j=1, \dots, r_n} \left( \mu^{(j)} + \sigma^{(j)}  Z_j  \right) \leq m_n^a\right) \mathbf{P}(((b_i^{(j)}),\mathcal{N}^{(j)}) = ((d_i^{(j)}),M^{(j)}) \forall j) \nonumber \\
\geq~& \sum_{\{a_n \leq m_n^a\}} \mathbf{Q}\left( \min_{j=1, \dots, r_n} \left( \mu^{(j)} + \sigma^{(j)}  Z_j  \right) \leq a_n \right) \mathbf{P}(((b_i^{(j)}),\mathcal{N}^{(j)}) = ((d_i^{(j)}),M^{(j)}) \forall j) \nonumber \\
\geq~& \Cr{c: from_appendix} \mathbf{P}(a_n \leq m_n^a, ((b_i^{(j)}),\mathcal{N}^{(j)}) \in \Upsilon^{(j)}~\forall j) \nonumber \\
\geq~&\frac{\Cr{c: from_appendix}}{4}. \label{eq: lower_fluctuation}
\end{align}
A similar argument using \eqref{eq: b_n_inequality} shows that
\begin{equation}\label{eq: upper_fluctuation}
\sum_{\Upsilon^{(1)} \times \dots \times \Upsilon^{(r_n)}}\mathbf{Q}\left( \min_{j=1, \dots, r_n} \left( \mu^{(j)} + \sigma^{(j)}  Z_j  \right) \geq m_n^b\right) \mathbf{P}(((b_i^{(j)}),\mathcal{N}^{(j)}) = ((d_i^{(j)}),M^{(j)}) \forall j) \geq \frac{\Cr{c: from_appendix}}{4}.
\end{equation}
Finally, we put \eqref{eq: lower_fluctuation} into \eqref{eq: here_we_are_now} to find
\[
\liminf_{n \to \infty} \mathbf{P}(\mathcal{T}_n \leq m_n^a) \geq \frac{\Cr{c: from_appendix}}{4}
\]
and similarly using \eqref{eq: upper_fluctuation} gives
\[
\liminf_{n \to \infty} \mathbf{P}(\mathcal{T}_n \geq m_n^b) \geq \frac{\Cr{c: from_appendix}}{4}.
\]
Because of \eqref{eq: tilde_variance_inequality} and \eqref{eq: median_choice}, this completes the proof of Prop.~\ref{prop: independent_fluctuations}.
\end{proof}

\section{Proofs of main results}\label{sec: proofs}

\subsection{Proof of Thm.~\ref{main}}

We will first bound fluctuations of $T_n^{\textnormal{tube}}$ by comparing it to a minimum of cylinder times. Define integers $K_n^{(1)}, \dots, K_n^{(r_n)} \in [1,n]$ such that $\sum_{j=1}^{r_n} K_n^{(j)} = n$, along with cylinders $\mathsf{C}^{(j)}$, passage times $T^{(j)}$, and minimum $\mathcal{T}_n$, all as in Sec.~\ref{sec: independent_cylinders}. Although in that section we used the i.i.d.~weights $(t_e)$, here we define them using the periodic weights $(t_e^{(n)})$. This does not change the distribution of $\mathcal{T}_n$ because the cylinders are contained in $[0,n] \times [0,n-1]$. For $\alpha_1,\alpha_2$ such that $3/4 < \alpha_1 < \alpha_2 < 1$, we assume that
\begin{equation}\label{eq: parameters}
K_n^{(j)} \in [n^{\alpha_2}, 2 n^{\alpha_2}]\text{ for all } j, \text{ so } r_n \leq n^{1-\alpha_2}.
\end{equation}
In addition, we shift the cylinders up by $\lfloor n^{\alpha_2}/2 \rfloor$, setting
\[
(K_n^{(j)})' = K_n^{(j)} + \left\lfloor \frac{n^{\alpha_2}}{2} \right \rfloor \text{ for } j =1 , \dots, r_n,
\]
with corresponding cylinders $(\mathsf{C}^{(j)})' = \mathsf{C}^{(j)} + \lfloor n^{\alpha_2}/2\rfloor \mathbf{e}_2$, passage times $(T^{(j)})'$, and minimum $\mathcal{T}_n'$. Again we use the periodic weights $(t_e^{(n)})$ instead of the i.i.d.~weights $(t_e)$. Observe that $\mathcal{T}_n$ and $\mathcal{T}_n'$ have the same distribution, but they are not independent. Last, define $\mathsf{T}_n = \min\{\mathcal{T}_n, \mathcal{T}_n'\},$ and let $\mathsf{A}_n$ be the event described in Prop.~\ref{conf}:
\begin{equation}\label{eq: A_n_def}
\mathsf{A}_n = \left\{ \text{for any geodesic }\Gamma_n \text{ for }T_n^{\textnormal{tube}}, ~\Gamma_n \subset x_0 + \left( \mathbb{Z} \times [-n^{\alpha_1}, n^{\alpha_1}]\right)\right\},
\end{equation}
where $x_0$ is the initial point of $\Gamma_n$. That proposition gives 
\begin{equation}\label{eq: proposition_consequence}
\mathbf{P}(\mathsf{A}_n) \geq 1 - e^{-n^b} \text{ for large }n.
\end{equation}

Any optimal path for $\mathsf{T}_n$ connects the sets $\{x : x \cdot \mathbf{e}_1 = 0\}$ and $\{x : x \cdot \mathbf{e}_1 = n\}$, so by \eqref{tubetor}, we have $T_n^{\textnormal{tube}} \leq \mathsf{T}_n$. We claim that if $n$ is large, then in fact
\begin{equation}\label{eq: split_it_up}
\text{on }\mathsf{A}_n, \text{ we have } T_n^{\textnormal{tube}} = \mathsf{T}_n.
\end{equation}
To see why this holds, consider an outcome in $\mathsf{A}_n$ with $n$ large and let $\Gamma_n$ be a geodesic for $T_n^{\textnormal{tube}}$. By periodicity of the weights, we may select $\Gamma_n$ so that its initial point $x_0 = (0,m_0)$ satisfies $m_0 \in  \left[ \frac{n^{\alpha_2}}{10}  , n+ \frac{ n^{\alpha_2}}{10} \right]$. If the set $[0,n] \times [m_0 - n^{\alpha_1}, m_0+n^{\alpha_1}]$ is not contained in any of the cylinders $\mathsf{C}^{(j)}$, then either $m_0 \geq n$ or the interval $[m_0 - n^{\alpha_1}, m_0+n^{\alpha_1}]$ must contain a number of the form $K_n^{(1)} + \dots + K_n^{(j_0)}$. In the first case, $[0,n] \times [m_0 - n^{\alpha_1}, m_0 + n^{\alpha_1}]$ is contained in $(\mathsf{C}^{(r_n)})'$, and in the second, it is contained in $(\mathsf{C}^{(j_0)})'$. In any case, since $\mathsf{A}_n$ occurs, $\Gamma_n$ must be contained in one of the cylinders $\mathsf{C}^{(j)}$ or $(\mathsf{C}^{(j)})'$, so it is an admissible path for the definition of the corresponding cylinder passage time $T^{(j)}$ or $(T^{(j)})'$. Therefore $T_n^{\textnormal{tube}} = T^{(n)}(\Gamma_n) \geq \mathsf{T}_n$.

Now we apply Prop.~\ref{prop: independent_fluctuations}. We estimate from \eqref{eq: parameters}
\[
\frac{1}{\sqrt{n}} \sum_{j=1}^{r_n} \frac{1}{\sqrt{K_n^{(j)}}} \leq n^{\frac{1-3\alpha_2}{2}} \to 0
\]
since $\alpha_2 > 1/3$, and
\[
\min_{j=1, \dots, r_n} \sqrt{\frac{n}{K_n^{(j)}(1+\log r_n)}} \geq \sqrt{\frac{n^{1-\alpha_2}}{2(1+ \log n^{1-\alpha_2})}}.
\]
Therefore we can find reals $a_n,b_n$ and a constant $\Cl[smc]{c: new_spread_constant}>0$ such that for large $n$,
\begin{equation}\label{eq: choice_of_a_n_b_n}
b_n - a_n \geq \Cr{c: new_spread_constant} \sqrt{\frac{n^{1-\alpha_2}}{\log n}}
\end{equation}
and
\begin{equation}\label{eq: cylinder_fluct}
\mathbf{P}(\mathcal{T}_n \geq b_n) \geq \Cr{c: new_spread_constant} \text{ and } \mathbf{P}(\mathcal{T}_n \leq a_n) \geq \Cr{c: new_spread_constant},
\end{equation}
with the same statements holding for $\mathcal{T}_n'$. Both $\mathcal{T}_n$ and $\mathcal{T}_n'$ are decreasing functions of the i.i.d.~weights $(t_e^{(n)})_{e \in E(n)^o}$, so by Harris's inequality,
% \cite[Thm.~2.15]{BLM},
\[
\mathbf{P}(\mathsf{T}_n \geq b_n) \geq \mathbf{P}(\mathcal{T}_n \geq b_n)^2 \geq \Cr{c: new_spread_constant}^2.
\]
We also have $\mathbf{P}(\mathsf{T}_n \leq a_n) \geq \mathbf{P}(\mathcal{T}_n \leq a_n) \geq \Cr{c: new_spread_constant}$. Last, we use \eqref{eq: proposition_consequence} and \eqref{eq: split_it_up} to get for large $n$
\[
\mathbf{P}(T_n^{\textnormal{tube}} \geq b_n) \geq \mathbf{P}(\mathsf{T}_n \geq b_n) - e^{-n^b} \geq \frac{\Cr{c: new_spread_constant}^2}{2}
\]
and similarly $\mathbf{P}(T_n^{\textnormal{tube}} \leq a_n) \geq \Cr{c: new_spread_constant}/2$. Therefore $(T_n^{\textnormal{tube}})$ has fluctuations at least of order $\sqrt{n^{1-\alpha_2}/\log n}$. Since $\alpha_2$ is an arbitrary number bigger than $3/4$, this completes the proof of the first statement of Thm.~\ref{main}.

We move to the other half of Thm.~\ref{main}. To bound fluctuations for $T_n^{\textnormal{sq}}$, we compare this quantity to $T_n^{\textnormal{tube}}$. Because of the result we just proved, if $\epsilon>0$, we can choose $\alpha_2$ above \eqref{eq: parameters} with $(1-\alpha_2)/2 > 1/8 - \epsilon$ such that the reals $a_n,b_n$ defined in \eqref{eq: choice_of_a_n_b_n} satisfy $b_n - a_n \geq n^{1/8- \epsilon}$ for all large $n$ and such that
\begin{equation}\label{eq: a_n_b_n_def_now}
\mathbf{P}(T_n^{\textnormal{tube}} \leq a_n) \geq \Cl[smc]{c: another_new_spread_constant} \text{ and } \mathbf{P}(T_n^{\textnormal{tube}} \geq b_n) \geq \Cr{c: another_new_spread_constant}
\end{equation}
for some $\Cr{c: another_new_spread_constant}>0$ (and similarly for $\mathcal{T}_n$). We first show that
\begin{equation}\label{eq: square_lower}
\liminf_{n \to \infty} \mathbf{P}(T_n^{\textnormal{sq}} \leq a_n) > 0.
\end{equation}
We will assume that the weights $(t_e)$ and $(t_e^{(n)})$ used in the definitions of $T_n^{\textnormal{sq}}$ and $T_n^{\textnormal{tube}}$ are coupled so that $t_e = t_e^{(n)}$ for all $e \in E(n)^o$. First, taking the definition of $\mathsf{A}_n$ from \eqref{eq: A_n_def}, and using \eqref{eq: a_n_b_n_def_now}, we obtain
\begin{equation}\label{eq: square_upper_step_1}
\liminf_{n \to \infty} \mathbf{P}(T_n^{\textnormal{tube}} \leq a_n,~\mathsf{A}_n) > 0.
\end{equation}
Let $\mathsf{E}_n$ be the event that some geodesic for $T_n^{\textnormal{tube}}$ has initial point in the interval $\{0\} \times [\lfloor n/10 \rfloor, \lfloor n/10 \rfloor + \lfloor 8n/10 \rfloor]$, and let $\mathsf{E}_n'$ be the event that some geodesic for $T_n^{\textnormal{tube}}$ has initial point in the interval $\{0\} \times [\lfloor n/2 \rfloor, \lfloor n/2\rfloor + \lfloor 8n/10\rfloor]$. Then by vertical translation invariance,
\begin{align*}
\mathbf{P}(T_n^{\textnormal{tube}} \leq a_n,~\mathsf{A}_n) \leq \mathbf{P}(T_n^{\textnormal{tube}} \leq a_n,~\mathsf{A}_n,\mathsf{E}_n)  &+ \mathbf{P}(T_n^{\textnormal{tube}} \leq a_n,~\mathsf{A}_n,\mathsf{E}_n') \\
&= 2\mathbf{P}(T_n^{\textnormal{tube}} \leq a_n,~\mathsf{A}_n,\mathsf{E}_n),
\end{align*}
so by \eqref{eq: square_upper_step_1},
\[
\liminf_{n \to \infty} \mathbf{P}(T_n^{\textnormal{tube}} \leq a_n,~\mathsf{A}_n,\mathsf{E}_n) > 0.
\]
However, if $n$ is large and $\mathsf{A}_n \cap \mathsf{E}_n$ occurs, then there is a geodesic $\Gamma_n$ for $T_n^{\textnormal{tube}}$ that uses only edges in $E(n)^o$, so it is an admissible path for the definition of $T_n^{\textnormal{sq}}$. Therefore $T_n^{\textnormal{tube}} = T^{(n)}(\Gamma_n) = T(\Gamma_n) \geq T_n^{\textnormal{sq}}$, and so 
\[
\liminf_{n \to \infty} \mathbf{P}(T_n^{\textnormal{sq}} \leq a_n) \geq \liminf_{n \to \infty} \mathbf{P}(T_n^{\textnormal{tube}} \leq a_n,~\mathsf{A}_n,\mathsf{E}_n) > 0.
\]
%This proves \eqref{eq: square_upper}.

To complete the proof, we must show
\begin{equation}\label{eq: square_upper}
\liminf_{n \to \infty} \mathbf{P}(T_n^{\textnormal{sq}} \geq b_n) > 0.
\end{equation}
To do this, we use three related passage times, all defined with the weights $(t_e)$. Let $T_n(1)$ be the minimum of $T(\Gamma)$ over all $\Gamma$ with vertices in $B(n)$, connecting $\{0\} \times [0,n]$ to $\{n\} \times [0,n]$, but not using edges with both endpoints in $[0,n] \times \{n\}$ (that is, using only edges in $E(n)^o$). Let $T_n(2)$ be the minimum of $T(\Gamma)$ over all $\Gamma$ with vertices in $B(n)$, connecting the same two sets, but now not using edges with both endpoints in $[0,n] \times \{0\}$. Last, let $T_n(3)$ be the minimum of $T(\Gamma)$ over all $\Gamma$ with vertices in $B(n)$, connecting $[0,n] \times \{0\}$ to $[0,n] \times \{n\}$, but not using edges with both endpoints in $\{n\} \times [0,n]$. All the $T_n(i)$ are decreasing functions of the weights $(t_e)$ and are identically distributed. Furthermore, using our coupling of $(t_e)$ and $(t_e^{(n)})$, we have $T_n(1) \geq T_n^\textnormal{tube}$. So by Harris's inequality and \eqref{eq: a_n_b_n_def_now},
\begin{equation}\label{eq: min_inequality}
\liminf_{n \to \infty} \mathbf{P}\left(\min_{i=1, 2,3} T_n(i) \geq b_n\right) \geq \liminf_{n \to \infty} \mathbf{P}(T_n(1) \geq b_n)^3 \geq \liminf_{n \to \infty} \mathbf{P}(T_n^{\textnormal{tube}} \geq b_n)^3 > 0.
\end{equation}
However, 
\begin{equation}\label{eq: last_domination}
\min_{i=1,2,3} T_n(i) \leq T_n^{\textnormal{sq}}
\end{equation} 
because of the following argument. Let $\sigma_n$ be a geodesic for $T_n^{\textnormal{sq}}$ (one exists because there are finitely many self-avoiding paths in $B(n)$). If $\sigma_n$ does not touch $[0,n] \times \{n\}$, then it is an admissible path for $T_n(1)$, so $T_n^\textnormal{sq} \geq T_n(1)$. If $\sigma_n$ does not touch $[0,n] \times \{0\}$, then similarly $T_n^\textnormal{sq} \geq T_n(2)$. Last, if $\sigma_n$ touches both of these sets, let $\sigma_n'$ be a segment of $\sigma_n$ that connects them. Because $\sigma_n$ is a geodesic for $T_n^\textnormal{sq}$ it a.s.~cannot use an edge with both endpoints in $\{n\} \times [0,n]$. Therefore $\sigma_n'$ cannot either, and so it is an admissible path for $T_n(3)$, implying that $T_n^\textnormal{sq} = T(\sigma_n) \geq T(\sigma_n') \geq T_n(3)$. In any of these three cases, \eqref{eq: last_domination} holds.

Finally, due to \eqref{eq: last_domination}, we have $\mathbf{P}(T_n^\textnormal{sq} \geq b_n) \geq \mathbf{P}(\min_{i=1,2,3} T_n(i) \geq b_n)$. This along with \eqref{eq: min_inequality} implies \eqref{eq: square_upper} and completes the proof of Thm.~\ref{main}.

\subsection{Proof of Cor.~\ref{corollary}}

The proof of Cor.~\ref{corollary} will use the objects defined in the last proof. Specifically, we take $\alpha_1,\alpha_2$ as above \eqref{eq: parameters}, the cylinder time $\mathcal{T}_n$, and reals $a_n,b_n$ defined in \eqref{eq: choice_of_a_n_b_n}. For these reals, we have shown above that we have, in addition to \eqref{eq: cylinder_fluct}, similar inequalities for $T_n^\textnormal{tube}$ and $T_n^\textnormal{sq}$. We will also use the event $\mathsf{A}_n$ defined in \eqref{eq: A_n_def}.

We first prove that for some $\Cl[smc]{c: last_new_spread_constant}>0$, we have for large $n$
\begin{equation}\label{eq: torus_upper}
\mathbf{P}(T_n^\textnormal{tube} \geq b_n') \geq \Cr{c: last_new_spread_constant},
\end{equation}
where 
\[
b_n' = \begin{cases}
2 b_{\frac{n}{2}} & \quad \text{if } n \text{ is even} \\
b_{\frac{n-1}{2}} + b_{\frac{n+1}{2}} & \quad \text{if } n \text{ is odd}.
\end{cases}
\]
Because $T_n^\textnormal{tor} \geq T_n^\textnormal{tube}$, this will also establish $\mathbf{P}(T_n^\textnormal{tor} \geq b_n') \geq \Cr{c: last_new_spread_constant}$.

To prove \eqref{eq: torus_upper}, first suppose that $n$ is even and consider the three squares
\[
\mathsf{S}_i = \left[ 0 , \frac{n}{2} \right]^2 + i \left\lfloor \frac{n}{3}\right\rfloor \mathbf{e}_2 \text{ for } i=1,2,3
\]
with corresponding square passage times $T_n^\textnormal{sq}(i)$ (all using the weights $(t_e^{(n)})$). We also define the shifted squares
\[
\bar{\mathsf{S}}_i = \left( \mathsf{S}_i + n \mathbf{e}_1 \right) + i \left\lfloor \frac{n}{3} \right\rfloor \mathbf{e}_2 \text{ for } i = 1,2,3
\]
along with their corresponding square passage times $\bar{T}_n^\textnormal{sq}(i)$ (again using $(t_e^{(n)})$). From the version of \eqref{eq: cylinder_fluct} for $T_n^{\textnormal{sq}}$ and Harris's inequality, we have
\[
\mathbf{P}\left(\min_{i=1,2,3} \min\left\{ T_n^{\textnormal{sq}}(i), \bar{T}_n^{\textnormal{sq}}(i)\right\} \geq b_{\frac{n}{2}}\right) \geq \Cr{c: new_spread_constant}^6.
\]
As a consequence of \eqref{eq: proposition_consequence}, $\mathbf{P}(\mathsf{A}_n) \to 1$, and so
\begin{equation}\label{eq: upper_almost_end_now}
\liminf_{n \to \infty} \mathbf{P}\left( \mathsf{A}_n, \min_{i=1,2,3}\left( T_n^{\textnormal{sq}}(i) + \bar{T}_n^{\textnormal{sq}}(i)\right) \geq b_n'\right) > 0.
\end{equation}
However, on the event in this probability, we have $T_n^{\textnormal{tube}} \geq b_n'$. Indeed, choose a geodesic $\Gamma_n$ for $T_n^{\textnormal{tube}}$. We may assume by periodicity that the initial point $x_0$ of $\Gamma_n$ is contained in the interval $\{0\} \times [n/6,n+n/6]$. Then if $n$ is large, there exists some $i$ such that
\[
x_0 + \left( [0,n] \times [-n^{\alpha_1},n^{\alpha_1}]\right) \subset \left(\mathsf{S}_i \cup \bar{\mathsf{S}}_i\right) .
\]
This means that $\Gamma_n$ contains an initial segment that is an admissible path for $T_n^{\textnormal{sq}}(i)$ and a final (disjoint) segment that is an admissible path for $\bar{T}_n^\textnormal{sq}(i)$, so $T_n^\textnormal{tube} = T^{(n)}(\Gamma_n) \geq T_n^{\textnormal{sq}}(i) + \bar{T}_n^\textnormal{sq}(i) \geq 2b_{n/2} = b_n'$. Because of \eqref{eq: upper_almost_end_now}, we then conclude that \eqref{eq: torus_upper} holds if $n$ is even. In the case that $n$ is odd, the squares are instead defined as
\[
\mathsf{S}_i = \left[ 0 , \frac{n-1}{2} \right]^2 + i \left\lfloor \frac{n}{3}\right\rfloor \mathbf{e}_2 \text{ for } i=1,2,3
\]
and
\[
\bar{\mathsf{S}}_i = \left[ \frac{n-1}{2},n\right] \times \left[ 0, \frac{n+1}{2}\right] + i \left\lfloor \frac{n}{3} \right\rfloor \mathbf{e}_2 \text{ for } i = 1,2,3,
\]
but the rest of the proof is the same.

The second half of the proof serves to show that if $\beta \in (0,1)$, there is $\Cl[smc]{c: small_tail}>0$ such that for large $n$,
\begin{equation}\label{eq: torus_lower}
\mathbf{P}(T_n^\textnormal{tor} \leq a_n') \geq \frac{\Cr{c: small_tail}}{n^{1 +\alpha_2-2\beta}},
\end{equation}
where 
\[
a_n' = \begin{cases}
2a_{\frac{n}{2}} + 8n^\beta & \quad \text{if } n \text{ is even} \\
a_{\frac{n-1}{2}} + a_{\frac{n+1}{2}} + 8 n^\beta & \quad \text{if } n \text{ if odd}.
\end{cases}
\]
We do this by approximately concatenating two low-weight paths in cylinders of height $n^{\alpha_1}$, one in the left half of the torus and one in the right. Again we will assume that $n$ is even; a similar argument works if $n$ is odd. 

First, in addition to the variable $\mathcal{T}_{n/2}$, we define $\bar{\mathcal{T}}_{n/2}$ as the corresponding minimum over the left-right times associated with shifted cylinders $\mathsf{C}^{(j)} + (n/2)\mathbf{e}_1$ and the same integer cylinder-heights $K_{n/2}^{(1)}, \dots, K_{n/2}^{(r_{n/2})}$ that satisfy equation \eqref{eq: parameters} with $n$ replaced by $n/2$. In these definitions, we use the periodic weights $(t_e^{(n)})$. By independence and \eqref{eq: cylinder_fluct}, we have
\[
\mathbf{P}\left(\max\left\{ \mathcal{T}_{\frac{n}{2}}, \bar{\mathcal{T}}_{\frac{n}{2}}\right\} \leq a_{\frac{n}{2}}\right) \geq \Cr{c: new_spread_constant}^2.
\]
Let $G_n$ be the event that there exists a path $\gamma$ with initial point $x_0$ in $\{0\} \times [0,n]$, all of whose edges intersect $(0,n/2) \times [0,n]$, with final point in $\{n/2\} \times [0,n]$, such that $T^{(n)}(\gamma) \leq a_{n/2}$ and $\gamma \subset x_0 + \left( \mathbb{Z} \times [-2(n/2)^{\alpha_2},2(n/2)^{\alpha_2}] \right)$. Similarly, define $\bar{G}_n$ as the event that there is a path $\bar{\gamma}$ with initial point $y_0$ in $\{n/2\} \times [0,n]$, all of whose edges intersect $(n/2,n) \times [0,n]$, with final point in $\{n\} \times [0,n]$, such that $T^{(n)}(\bar{\gamma}) \leq a_{n/2}$ and $\bar{\gamma} \subset y_0 + \left( \mathbb{Z} \times [-2(n/2)^{\alpha_2},2(n/2)^{\alpha_2}]\right)$. By \eqref{eq: parameters}, any optimal path in the definition of $\mathcal{T}_{n/2}$ (or $\bar{\mathcal{T}}_{n/2}$) satisfies the properties in the definition of $G_n$ (or $\bar{G}_n$), so for large $n$
\begin{equation}\label{eq: independent_times}
\mathbf{P}(G_n \cap \bar{G}_n) \geq \Cr{c: new_spread_constant}^2
\end{equation}

On the event $G_n \cap \bar{G}_n$, we will build an admissible path for $T_n^\textnormal{tor}$ consisting of $\gamma$, $\bar{\gamma}$, and vertical line segments. To do this effectively, we need the endpoints of $\gamma$ and $\bar{\gamma}$ to be close to each other. For this purpose, we must introduce a few definitions. First, let $\beta \in (0,1)$ and define a family of intervals
\[
I_1^1 = \{0\} \times [0,n^\beta], I_2^1 = \{0\} \times [ n^\beta,2n^\beta], \dots, I_p^1 = \{0\} \times \left[ (\lfloor n^{1-\beta}\rfloor - 1) n^\beta, \lfloor n^{1-\beta}\rfloor n^\beta\right],
\]
and $I_{p+1}^1 = \{0\} \times [ \lfloor n^{1-\beta}\rfloor n^\beta, n]$. We need shifted intervals as well, so we set
\[
I_i^2 = I_i^1 + \frac{n}{2} \mathbf{e}_1 \text{ and } I_i^3 = I_i^1 + n\mathbf{e}_1.
\]
Also define $i_1$ as the minimal value of $i$ for which there is a path $\gamma$ satisfying the conditions of the definition of $G_n$ such that its initial point $x_0$ is in $I_i^1$, and similarly define $i_3$ as the minimal value of $i$ for which there is a path $\bar{\gamma}$ satisfying the conditions in the definition of $\bar{G}_n$ such that its initial point $y_0$ is in $I_i^2$. Set $i_2$ to be the minimal value of $i$ for which there is a path $\gamma$ satisfying the conditions of the definition of $G_n$ such that its initial point $x_0$ is in $I_{i_1}^1$ but its final point is in $I_i^2$. Finally, set $i_4$ to be the minimal value of $i$ for which there is a path $\bar{\gamma}$ satisfying the conditions of the definition of $\bar{G}_n$ such that its initial point $y_0$ is in $I_{i_3}^2$ and its final point is in $I_i^3$. If $G_n$ occurs, then $i_1$ and $i_2$ are defined, and if $\bar{G}_n$ occurs, then $i_3$ and $i_4$ are defined. Whenever any $i_j$ is not defined, we set it to be $+\infty$. 

From \eqref{eq: independent_times}, we have
\begin{equation}\label{eq: i_j_are_finite}
\mathbf{P}(i_j < \infty \text{ for } j=1,\dots, 4) \geq \Cr{c: new_spread_constant}^2.
\end{equation}
Furthermore, by independence,
\begin{align}
&\mathbf{P}(i_1=i_4 < \infty, i_2=i_3 < \infty) \nonumber \\
=~& \sum_{i=1}^{p+1} \left[ \sum_{j=1}^{p+1} \mathbf{P}(i_2=i_3 = j \mid i_1 = i_4 = i)\right] \mathbf{P}(i_1=i_4=i) \nonumber \\
=~& \sum_{i=1}^{p+1} \left[ \sum_{j=1}^{p+1} \mathbf{P}(i_2= j \mid i_1 = i) \mathbf{P}(i_3=j \mid i_4 = i)\right] \mathbf{P}(i_1=i) \mathbf{P}(i_4=i). \label{eq: conditional_breakdown}
\end{align}
Observe that if $i_1=i$, then $\gamma$ begins at $x_0$ in the interval $I_i^1$ but must end at an interval in $x_0 + \left( \mathbb{Z} \times [-2(n/2)^{\alpha_2},2(n/2)^{\alpha_2}] \right)$. Therefore $i_2$ can take only take values in a set $\mathfrak{S}_i$ which has cardinality at most $8n^{\alpha_2-\beta}$ if $n$ is large. Using Jensen's inequality and symmetry, the inner sum is
\begin{align*}
\sum_{j \in \mathfrak{S}_i} \mathbf{P}(i_2=j \mid i_1=i)\mathbf{P}(i_3=j \mid i_4=i) &= \#\mathfrak{S}_i \cdot \frac{1}{\#\mathfrak{S}_i} \sum_{j\in \mathfrak{S}_i} \mathbf{P}(i_2=j \mid i_1=i)^2 \\
&\geq \frac{1}{\#\mathfrak{S}_i} \left( \sum_{j \in \mathfrak{S}_i} \mathbf{P}(i_2=j \mid i_1=i)\right)^2 \\
&\geq \left(8 n^{\alpha_2-\beta}\right)^{-1}.
\end{align*}
By a similar argument, and now using \eqref{eq: i_j_are_finite}, we obtain
\[
\sum_{i=1}^{p+1} \mathbf{P}(i_1=i) \mathbf{P}(i_4=i) \geq \frac{1}{p+1} \mathbf{P}(i_1<\infty)^2 \geq \frac{\Cr{c: new_spread_constant}^4}{p+1}.
\]
Putting these back in \eqref{eq: conditional_breakdown} and using $p+1 \leq 2n^{1-\beta}$ gives
\begin{equation}\label{eq: matching_ends_bound}
\mathbf{P}(i_1=i_4<\infty, i_2=i_3<\infty) \geq \frac{\Cr{c: new_spread_constant}^4}{16n^{1 +\alpha_2-2\beta}}.
\end{equation}

On the event in the probability of \eqref{eq: matching_ends_bound}, the endpoints of $\gamma$ and $\bar{\gamma}$ are within $n^{\beta}$ of each other; now we will connect them with vertical segments. So let $H_n$ be the event that any vertex self-avoiding path $\pi$ with at most $2n^\beta$ many edges and which is contained in $\cup_{i=1}^3 \left( \{in/2\} \times [0,n]\right)$ satisfies $T^{(n)}(\pi) \leq 4n^\beta$. By the argument of \eqref{eq: exponential_markov}, any such $\pi$ satisfies
\[
\mathbf{P}(T^{(n)}(\pi) \geq 4n^\beta) \leq e^{-4\eta n^{\beta}} \left( \mathbf{E}e^{\eta t_e}\right)^{\#\pi} \leq e^{-\Cl[smc]{c: last_exponential}n^{\beta}}
\]
for some $\Cr{c: last_exponential}>0$, and so a union bound produces for large $n$
\[
\mathbf{P}(H_n) \geq 1 - 3(n+1) \cdot 2n^\beta \cdot e^{-\Cr{c: last_exponential}n^{\beta}} \geq 1 - e^{-\frac{\Cr{c: last_exponential}}{2} n^{\beta}}.
\]
So, if $n$ is large, we can combine this with \eqref{eq: matching_ends_bound} for
\begin{equation}\label{eq: almost_at_the_end_broseph}
\mathbf{P}(H_n, i_1=i_4<\infty, i_2=i_3<\infty) \geq \frac{\Cr{c: new_spread_constant}^4}{32n^{1 +\alpha_2-2\beta}}.
\end{equation}
Now consider an outcome for which $H_n$ occurs, and also $i_1=i_4<\infty$, and $i_2=i_3 < \infty$. Then we may produce a path $\Gamma$ by starting with $\gamma$, following the line $\{n/2\} \times [0,n]$ from the final point of $\gamma$ to the initial point $y_0$ of $\bar{\gamma}$, traversing $\bar{\gamma}$, and then following the line $\{n\} \times [0,n]$ from the final point of $\bar{\gamma}$ to $x_0 + n\mathbf{e}_1$, where $x_0$ is the initial point of $\gamma$. In this way we produce an admissible path for $T_n^{\textnormal{tor}}$, and $T^{(n)}(\pi) \leq 2(a_{n/2} + 4n^\beta)$. Together with \eqref{eq: almost_at_the_end_broseph}, this implies \eqref{eq: torus_lower}.

Having shown both inequalities \eqref{eq: torus_upper} and \eqref{eq: torus_lower}, we complete the proof of Cor.~\ref{corollary}. For concreteness, we again assume that $n$ is even. For our $\alpha_2>3/4$, let $\delta>0$ and set $\beta = 1/2 - \alpha_2/2 - \delta$. Then by \eqref{eq: choice_of_a_n_b_n}, for large $n$,
\[
b_n' - a_n' = 2 \left( b_{\frac{n}{2}} - a_{\frac{n}{2}}\right) - 16n^\beta \geq 2\Cr{c: new_spread_constant} \sqrt{\frac{n^{1-\alpha_2}}{\log n}} - 16 n^{\frac{1-\alpha_2}{2} - \delta} \geq \Cr{c: new_spread_constant} \sqrt{\frac{n^{1-\alpha_2}}{\log n}}.
\]
Furthermore, by \eqref{eq: torus_upper} and \eqref{eq: torus_lower}, for large $n$,
\[
\mathbf{P}(T_n^\textnormal{tor} \leq a_n') \geq \frac{\Cr{c: small_tail}}{n^{2(\alpha_2+\delta)}}
\]
and
\[
\mathbf{P}(T_n^\textnormal{tor} \geq b_n') \geq \frac{\Cr{c: small_tail}}{n^{2(\alpha_2+\delta)}}.
\]
For any random variable $Y$ with finite mean satisfying $\mathbf{P}(Y \leq a) \geq c$ and $\mathbf{P}(Y \geq a) \geq c$ for reals $a \leq b$, one has $\mathbf{E}|Y-\mathbf{E}Y|^k \geq ((b-a)/2)^k c$. Applying this to $T_n^\textnormal{tor}$ gives
\[
\mathbf{E}|T_n^\textnormal{tor} - \mathbf{E}T_n^\textnormal{tor}|^k \geq \left( \frac{\Cr{c: new_spread_constant}}{2} \sqrt{\frac{n^{1-\alpha_2}}{\log n}}\right)^k \cdot \frac{\Cr{c: small_tail}}{n^{2(\alpha_2+\delta)}}.
\]
Last, taking $\alpha_2 = 3/4 + \delta$ gives for large $n$
\[
\mathbf{E}|T_n^\textnormal{tor} - \mathbf{E}T_n^\textnormal{tor}|^k \geq \left( \frac{\Cr{c: new_spread_constant}}{2} \cdot \frac{n^{\frac{1}{8} - \frac{\delta}{2}}}{\sqrt{\log n}} \right)^k \cdot \frac{\Cr{c: small_tail}}{n^{\frac{3}{2} + 4 \delta}} \geq \Cr{c: small_tail} n^{\left( \frac{1}{8} - \delta\right)k - \frac{3}{2} - 4 \delta}.
\]
This implies Cor.~\ref{corollary} and completes the proof.

\appendix

\section{Gaussian fluctuation lemmas} \label{app:variance}

This section provides the fluctuation result for normal random variables used to justify \eqref{eq: a_n_b_n_choice} and \eqref{eq: spread} in the proof of Prop.~\ref{prop: independent_fluctuations}. It relates the fluctuations of the minimum of independent normal variables with different means and variances to the fluctuations of the minimum of i.i.d.~normal variables. For notational convenience, the result is stated for the maximum, though by symmetry it holds for the minimum.
\begin{theorem}\label{thm: gaussian_fluctuations}
There exists a universal constant $\Cl[smc]{c: real_appendix}>0$ such that the following holds. Let $Z_1, \dots, Z_n$ be independent normal variables with variances $\sigma_1^2, \dots, \sigma_n^2$ and set $\sigma = \min_{i=1, \dots, n} \sigma_i$. There exist reals $a_n,b_n$ with
\begin{equation}\label{eq: iid_fluctuations}
b_n - a_n = \frac{\sigma}{\sqrt{1 + \log n}}
\end{equation}
such that
\[
\mathbf{P}\left( \max_{i=1, \dots, n} Z_i \leq a_n\right) \geq \Cr{c: real_appendix} \text{ and } \mathbf{P}\left( \max_{i=1, \dots, n} Z_i \geq b_n \right) \geq \Cr{c: real_appendix}
\]
\end{theorem}
By \cite[Ex.~3.2.3]{durrettbook}, the right side of \eqref{eq: iid_fluctuations} is the order of the fluctuations of i.i.d.~normal variables with variance $\sigma^2$. Before the proof of Thm.~\ref{thm: gaussian_fluctuations}, we state three lemmas. The first establishes a certain logconcavity property of the normal distribution.

\begin{lemma}\label{logcon}
Let $\Phi$ be the distribution function of the standard normal distribution and $f$ be its density function. Then $f/\Phi$ is log-concave.   
\end{lemma}
\begin{proof}
We must show that $(\log (f/\Phi))'' \leq 0$, but since $(\log f)'' \equiv -1$, this reduces to $(\log \Phi(x))'' \geq -1,$ or
\begin{equation}\label{logcon2}
f'(x) \Phi(x) -f^2(x) +\Phi^2(x) \geq 0.
\end{equation}
So let $F=f'\Phi-f^2+\Phi^2.$ Because $\lim_{x \to -\infty} F(x) = 0$, it will suffice to show that $F'(x) \geq 0$ for all $x$. Using $f'(x)=-xf(x),$ a computation gives
\begin{equation*}
F'(x)=f(x) [(1+x^2)\Phi(x) +x f(x)],
 \end{equation*} 
so we can just prove that $H(x) \geq 0$ for all $x$, where $H(x) = (1+x^2)\Phi(x) + xf(x)$. We observe that
\[
H'(x) = 2x \Phi(x) +(1+x^2)f(x)+f(x)+xf'(x)
%& = 2x \Phi(x) + f(x) +x^2f(x) +f(x) -x^2f(x) \notag  \\
= 2(x \Phi(x) + f(x) ).
\]
and
\[
H''(x)=2 [\Phi(x) +x f(x) -x f(x)]=2 \Phi(x) \geq 0.
\]
From the first display, we obtain $\lim_{x \to - \infty} H'(x) = 0$, and combining this with the second display, we obtain $H'(x) \geq 0$, for all $x$. Last, since $\lim_{x \to -\infty} H(x) = 0$, this implies $H(x) \geq 0$ for all $x$, and so \eqref{logcon2} follows.
\end{proof}

The next lemma tells us how the quantiles of $\max_i Z_i$ change as we shift the $Z_i$. Although it is stated for normal variables, it holds for more general distributions.
\begin{lemma}\label{lem: derivative}
 Let $Z_1, \dots, Z_n$ be independent normal random variables with distribution functions $F_1, \dots, F_n$ and densities $f_1, \dots, f_n$. For $a \in \mathbb{R}$, set 
 \[
 F_{max}(z,a) = \mathbf{P}(\max\{Z_1+a,Z_2,\dots, Z_n\} \leq z),
 \]
 and for $t \in (0,1)$ define $F_{max}^{-1}(t,a)$ as the unique real such that
 \begin{equation}\label{eq: inverse_def}
 F_{max}(F_{max}^{-1}(t,a),a) = t.
 \end{equation}
Then
\[
 \frac{\partial}{\partial a} F_{max}^{-1}(t,a) = \frac{\frac{f_1(z-a)}{F_1(z-a)}}{\frac{f_1(z-a)}{F_1(z-a)} + \sum_{i=2}^n \frac{f_i(z)}{F_i(z)}}\bigg|_{z = F_{max}^{-1}(t,a)},
\]
\end{lemma}
\begin{proof} 
%We first write $F_{max}$ explicitly.
%\[
%F_{max}(z) = \mathbf{P}\left( \max_i Z_i \leq z\right) = \prod_{i=1}^n \mathbf{P}(Z_i \leq z) = \prod_{i=1}^n F_i(z - a_i),
%\]
%and wish to compute derivatives of the inverse of $F_{max}$ relative to the parameters $a_i$. To do this, let us make the dependency on the $a_i$ as explicit as possible. Think of $a_2, \dots, a_n$ as fixed and $a_1$ as variable. Then write $F_{max}(z)$ as $F_{max}(z,a_1)$ to emphasize the dependency on $a_1$, and let $F_{max}^{-1}(t,a_1)$ be the unique real such that
%\[
%F_{max}(F_{max}^{-1}(t,a_1),a_1) = t \text{ for } t \in (0,1).
%\]
%We will write $a$ for $a_1$. 
Fix $t \in (0,1)$ and define $H: \mathbf{R} \to \mathbf{R}^2$ by $H(a) = (F_{max}^{-1}(t,a),a)$, so that \eqref{eq: inverse_def} becomes $F_{max}(H(a)) = t$. By the chain rule,
%\[
%\left( \frac{\partial}{\partial z} F_{max}(z,a) ~~~\frac{\partial}{\partial a} F_{max}(z,a)\right)\bigg|_{(z,a) = H(t,a)} 
%\left( \begin{array}{cc} 
%\frac{\partial}{\partial t} F_{max}^{-1}(t,a) & \frac{\partial}{\partial a} F_{max}^{-1}(t,a) \\
%0 & 1
%\end{array}\right) = (1~~~0)
%\]
%Considering the second coordinate of the product,
\[
 \frac{\partial}{\partial z} F_{max}(z,a)\bigg|_{(z,a) = H(a)} \frac{\partial}{\partial a} F_{max}^{-1}(t,a) + \frac{\partial}{\partial a} F_{max}(z,a)\bigg|_{(z,a) = H(a)}  = 0,
 \]
 which implies
 \begin{equation}\label{eq: deriv_formula}
 \frac{\partial}{\partial a} F_{max}^{-1}(t,a) = - \frac{\frac{\partial}{\partial a} F_{max}(z,a)}{\frac{\partial}{\partial z} F_{max}(z,a)}\bigg|_{(z,a) = H(a)}.
 \end{equation}
 Using $F_{max}(z,a) = F_1(z-a) \prod_{i=2}^n F_i(z)$, these derivatives are
 \[
 \frac{\partial}{\partial a} F_{max}(z,a) = - f_1(z-a)\prod_{i=2}^n F_i(z),
\]
and
\[
\frac{\partial}{\partial z} F_{max}(z,a) = F_{max}(z,a)\left( \frac{f_1(z-a)}{F_1(z-a)} + \sum_{i=2}^n \frac{f_i(z)}{F_i(z)}\right).
\]
Place these back into \eqref{eq: deriv_formula} to complete the proof.
%We conclude that the form of the derivative is
%\[
% \frac{\partial}{\partial a} F_{max}^{-1}(t,a) = \frac{\frac{f_1(z-a)}{F_1(z-a)}}{\frac{f_1(z-a)}{F_1(z-a)} + \sum_{i=2}^n \frac{f_i(z-a_i)}{F_i(z-a_i)}}\bigg|_{z = F_{max}^{-1}(t,a)},
% \]
%and reverting back to using $a_1$, we get the statement of the lemma.
\end{proof}

Our last lemma uses the previous two in its proof, and implies Thm.~\ref{thm: gaussian_fluctuations} in the case that the $Z_i$ have the same variances but different means. The result is stronger, and is stated in terms of the dispersive order.

\begin{definition}
The random variable $X,$ with distribution function $F,$ is said to be less dispersed that the random variable $Y,$ with distribution function $G$, if
\begin{equation*}
 F^{-1}(b)- F^{-1}(a) \leq G^{-1}(b)-G^{-1}(a) \quad \textnormal{ for all } 0 < a \leq b < 1.
 \end{equation*} 
\end{definition}
If $X$ is less dispersed than $Y,$ then $\textnormal{Var} \, X \leq \textnormal{Var} \,Y$ provided that $Y$ has finite variance. See, for instance, \cite[Sec.~3.B]{SS07} for this and more facts on the dispersive order. Moreover, we observe that if $X_n$ is less dispersed than $Y_n$ for all $n,$ then $\{X_n\}_{n=1}^{\infty}$ has fluctations of lower order than $\{Y_n\}_{n=1}^{\infty}.$ 

The following result on fluctuations of the maximum of shifted normal variables holds for more general distributions. The proof only uses Lem.~\ref{lem: derivative} and the logconcavity property of normals from Lem.~\ref{logcon}.

\begin{lemma} \label{logcondisp}
Let $X_1, \dots, X_n$ be i.i.d. normal random variables. Suppose $a_1, \dots, a_n$ are reals and let $Z_i=X_i+a_i$ for all $i.$ Then $\max_i X_i$ is less dispersed than $\max_i Z_i.$
\end{lemma}

\begin{proof}
We may assume that $a_1 \leq a_2 \leq \dots \leq a_n$. Write $F$ and $f$ for the common distribution function and density of the $X_i$, so that the distribution function and density of $Z_i$ is $F(z-a_i)$ and $f(z-a_i)$. Define $F_{max}^{-1}(t,a_1, \dots, a_n)$ analogously to \eqref{eq: inverse_def}: it is the unique real such that 
\[
\mathbf{P}\left(\max_{i=1, \dots, n} (a_i + X_i) \leq F_{max}^{-1}(t, a_1, \dots, a_n)\right) = t.
\] 
We aim to prove that for $t_1<t_2$, one has
\begin{equation}\label{eq: quantiles_to_prove}
F_{max}^{-1}(t_2,a_1, \dots, a_n) - F_{max}^{-1}(t_1,a_1, \dots, a_n) \geq F_{max}^{-1}(t_2, a_n, \dots, a_n) - F_{max}^{-1}(t_1, a_n, \dots, a_n).
\end{equation}
Putting
\[
\psi(a_1, \dots, a_n) = F_{max}^{-1}(t_2,a_1, \dots, a_n) - F_{max}^{-1}(t_1,a_1, \dots, a_n),
\]
Lemma~\ref{lem: derivative} implies that
\[
\frac{\partial}{\partial a_k} \psi(a_1, \dots, a_n) = \frac{\frac{f(z-a_k)}{F(z-a_k)}}{\sum_{i=1}^n \frac{f(z-a_i)}{F(z-a_i)}}\bigg|_{z = F_{max}^{-1}(t_2,a_1, \dots, a_n)} - \frac{\frac{f(z-a_k)}{F(z-a_k)}}{\sum_{i=1}^n \frac{f(z-a_i)}{F(z-a_i)}}\bigg|_{z = F_{max}^{-1}(t_1,a_1, \dots, a_n)}.
\]

We claim that if $a_1 = \dots = a_k \leq a_{k+1} \leq \dots \leq a_n$, then
\begin{equation}\label{eq: derivative_claim}
\frac{\partial}{\partial a_k} \psi(a_1, \dots, a_n) \leq 0.
\end{equation}
This will imply \eqref{eq: quantiles_to_prove} after integrating the derivative of $\psi$ along the concatenation of straight line segments connecting the points $(a_1,a_2 \dots, a_n)$, $(a_2,a_2,a_3, \dots, a_n)$, $(a_3,a_3,a_3,a_4, \dots, a_n)$, $\dots, (a_n, \dots, a_n)$ in order. To simplify notation, put $z_i = F_{max}^{-1}(t_i,a_1, \dots, a_n)$, so that $z_2 \geq z_1$, let $g = f/F$, and write
\begin{align}
\frac{\frac{f(z_2-a_k)}{F(z_2-a_k)}}{\sum_{i=1}^n \frac{f(z_2-a_i)}{F(z_2-a_i)}} &= \frac{g(z_2-a_k)}{\sum_{i=1}^n g(z_2-a_i)} \nonumber \\
&= \frac{g(z_1-a_k) \frac{g(z_2-a_k)}{g(z_1-a_k)}}{\sum_{i=1}^n \left[ g(z_1-a_i)\frac{g(z_2-a_k)}{g(z_1-a_k)}\right] + \sum_{i=1}^n \left[ g(z_1-a_i) \left( \frac{g(z_2-a_i)}{g(z_1-a_i)} - \frac{g(z_2-a_k)}{g(z_1-a_k)} \right) \right]} \label{lograt}
\end{align}
By Lem.~\ref{logcon}, $g$ is logconcave, so because $a_i \geq a_k$ for $i=1, \dots, n$, we have
\[
\begin{split}
\frac{g(z_2-a_i)}{g(z_1-a_i)} = \exp\left( \log g(z_2-a_i) - \log g(z_1-a_i)\right) &\geq\exp\left( \log g(z_2-a_k) - \log g(z_1-a_k)\right) \\
&= \frac{g(z_2-a_k)}{g(z_1-a_k)},
\end{split}
\]
and \eqref{lograt} above becomes 
\[
\frac{\frac{f(z_2-a_k)}{F(z_2-a_k)}}{\sum_{i=1}^n \frac{f(z_2-a_i)}{F(z_2-a_i)}} = \frac{g(z_2-a_k)}{\sum_{i=1}^n g(z_2-a_i)} \leq \frac{g(z_1-a_k) \frac{g(z_2-a_k)}{g(z_1-a_k)}}{\sum_{i=1}^n \left[ g(z_1-a_i)\frac{g(z_2-a_k)}{g(z_1-a_k)}\right]} = \frac{\frac{f(z_1-a_k)}{F(z_1-a_k)}}{\sum_{i=1}^n \frac{f(z_1-a_i)}{F(z_1-a_i)}}.
\]
This proves \eqref{eq: derivative_claim}.
\end{proof}

Given the preparatory lemmas above, we can now prove Thm.~\ref{thm: gaussian_fluctuations}.
\begin{proof}[Proof of Thm.~\ref{thm: gaussian_fluctuations}]

By \cite[Ex.~3.2.3]{durrettbook}, there is an absolute constant $\Cl[smc]{c: another_real_appendix}>0$ such that if $W_1, W_2, \dots$ are i.i.d.~normal variables with mean zero and variance $\sigma^2$, then there are reals $a_n,b_n$ such that
\[
b_n - a_n = \frac{\sigma}{\sqrt{1+\log n}}
\]
and both
\[
\mathbf{P}\left( \max_{i=1, \dots, n} W_i \leq a_n\right) \geq \Cr{c: another_real_appendix} \text{ and } \mathbf{P}\left( \max_{i=1, \dots, n} W_i \geq b_n\right) \geq \Cr{c: another_real_appendix}.
\]

%\end{proof}
%
%
%
%{\color{red} Keep 'er goin'}
%
%\begin{lemma} \label{varlemma}
%Let $X_1, \dots, X_n$ be independent normal random variables with variances $\sigma_1, \dots, \sigma_n,$ and let $Z_1, \ldots, Z_n$ be i.i.d. normal random variables with variance $\min_i \sigma_i.$ Then  $\min_i Z_i$ is less dispersed than $\min_i X_i.$
%\end{lemma}
%
%\pf 
%By Lemma \ref{logcon} and Lemma \ref{logcondisp}, and the symmetry of the maximum and the minimum statistics for a set of i.i.d. normal distributions, we know that if $X_1, \dots, X_n$ are independent normals with variance 1 but possibly different means, then $\min_i X_i$ is more dispersed than $\min_i Z_i$, where $Z_1, \dots, Z_n$ are i.i.d. standard normal. From this it should follow that there exist two reals $a_n \leq b_n$ (depending on the $X_i$'s) and a universal constant $c > 0$ such that
%\[
%b_n - a_n = \frac{c}{\sqrt{\log n}}
%\]
%and
%\[
%\mathbf{P}\left(\min_i X_i \geq b_n\right) \geq c,
%\]
%\[
%\mathbf{P}\left(\min_i X_i \leq a_n\right) \geq c.
%\]

We consider $W_i' + V_i$, where $W_i' = W_i + \mathbf{E}Z_i$, and the $V_i$ are independent normal random variables (and independent of $(W_i)$) with $\mathrm{Var}~V_i = \sigma_i^2-\sigma^2$. Then $(Z_i)$ and $(W_i'+V_i)$ have the same distribution. Conditional on $\vec{V} = (V_1, \dots, V_n)$, the variables $W_i'+V_i$ are independent normals with variance $\sigma^2$ and possibly different means $V_i + \mathbf{E}Z_i$. Lem.~\ref{logcondisp} therefore gives that they (conditionally) are more dispersed than the sequence $W_1, \dots, W_n$. Therefore we can find 
%
%
%
%
%
%We claim that this statement remains true if the $X_i$'s also have different variances $\sigma_i$ satisfying
%\[
%1 = \sigma_1 \leq \sigma_i \text{ for all }i.
%\]
%So let $X_1, \dots, X_n$ be independent normals with means $g_i$ and variances $\sigma_i^2$ satisfying this assumption, and, as before, observe that $(X_1, \dots, X_n)$ is equal in distribution to $(U_1 + V_1, \dots, U_n + V_n)$, where the $U_i$'s and $V_i$'s are all independent normals with $\mathbf{E}U_i = g_i$, $\text{Var}~U_i = \sigma_1^2 = 1$, $\mathbf{E}V_i = 0$, and $\text{Var}~V_i = \sigma_i^2 - \sigma_1^2 = \sigma_i^2 - 1$. It will then suffice to show the above statement for $\min_i (U_i + V_i)$ in place of $\min_i X_i$. Conditional on $\vec V = (V_1, \dots, V_n)$, the variables $U_i + V_i$ are just independent normals with variance 1 and possibly different means $g_i + V_i$. So we can apply the statement from above, and this produces 
reals $a_n(\vec V)$ and $b_n(\vec V)$ such that
\[
b_n(\vec V) - a_n(\vec V) = \frac{\sigma}{\sqrt{1+\log n}}
\]
and
\[
\mathbf{P}\left(\max_{i=1, \dots, n} (W_i' + V_i) \geq b_n(\vec V) \mid \vec V\right) \geq \Cr{c: another_real_appendix},
\]
\[
\mathbf{P}\left(\max_{i=1,\dots, n} (W_i' + V_i) \leq a_n(\vec V) \mid \vec V\right) \geq \Cr{c: another_real_appendix}.
\]
We can choose $a_n(\vec V)$ and $b_n(\vec V)$ as Borel measurable functions of the vector $\vec V$, for instance by selecting them to be quantiles of the conditional distribution of $\max_i (W_i' + V_i)$ given $\vec V$. We now choose reals $a_n$ and $b_n$ depending on the distributions of these $a_n(\vec V)$ and $b_n(\vec V)$ as
\[
a_n = \sup\left\{x \in \mathbf{R} : \mathbf{P}(a_n(\vec V) \leq x) \leq 1/2\right\},
\]
\[
b_n = a_n + \frac{\sigma}{2\sqrt{1+\log n}}.
\]
We observe with this definition that, since the distribution function of $a_n(\vec V)$ is right-continuous, we have
\[
\mathbf{P}(a_n(\vec V) \leq a_n) \geq 1/2.
\]
Furthermore,
\[
\mathbf{P}(b_n(\vec V) \geq b_n) \geq \mathbf{P}(a_n(\vec V) \geq a_n) = 1 - \mathbf{P}(a_n(\vec V) < a_n) \geq 1/2.
\]

Then
\begin{align*}
\mathbf{P}\left(\max_{i=1, \dots, n} (W_i' + V_i) \leq a_n\right) &= \mathbf{E}\left[ \mathbf{P}\left(\max_{i=1, \dots, n} (W_i' + V_i) \leq a_n \mid \vec V\right) \right] \\
&\geq \mathbf{E}\left[ \mathbf{P}\left(\max_{i=1, \dots, n} (W_i' + V_i) \leq a_n(\vec V), a_n(\vec V) \leq a_n \mid \vec V\right) \right] \\
&= \mathbf{E}\left[ \mathbf{P}\left(\max_{i=1, \dots, n} (W_i' + V_i) \leq a_n(\vec V) \mid \vec V\right) \mathbf{1}_{\{a_n(\vec V) \leq a_n\}} \right] \\
&\geq \Cr{c: another_real_appendix} \mathbf{P}(a_n(\vec V) \leq a_n) \\
&\geq \frac{\Cr{c: another_real_appendix}}{2},
\end{align*}
and
\begin{align*}
\mathbf{P}\left(\max_{i=1, \dots, n} (W_i' + V_i) \geq b_n\right) &= \mathbf{E}\left[ \mathbf{P}\left(\max_{i=1, \dots, n} (W_i' + V_i) \geq b_n \mid \vec V\right) \right] \\
&\geq \mathbf{E}\left[ \mathbf{P}\left(\max_{i=1, \dots, n} (W_i' + V_i) \geq b_n(\vec V), b_n(\vec V) \geq b_n \mid \vec V\right) \right] \\
&= \mathbf{E}\left[ \mathbf{P}\left(\max_{i=1, \dots, n} (W_i' + V_i) \geq b_n(\vec V) \mid \vec V\right) \mathbf{1}_{\{b_n(\vec V) \geq b_n\}} \right] \\
&\geq \Cr{c: another_real_appendix} \mathbf{P}(b_n(\vec V) \geq b_n) \\
&\geq \frac{\Cr{c: another_real_appendix}}{2}.
\end{align*}
Therefore the original claim has been proved after replacing $\Cr{c: another_real_appendix}$ by $\Cr{c: another_real_appendix}/2$.
\end{proof}

\bigskip
\noindent
{\bf Acknowledgements.} The Research of M.D. is supported by an NSF CAREER grant and NSF grant DMS-2054559, and the research of C.H. is supported by Simons Grant $\#$524678.

\end{document}